\documentclass[11pt]{article}
\usepackage{bbm}
\usepackage{amsfonts}
\usepackage{amssymb}
\usepackage{stmaryrd}
\usepackage{bm}
\usepackage{bbding}
\usepackage{amsmath}
\usepackage[colorlinks,
            linkcolor=blue,
            anchorcolor=blue,
            citecolor=blue
            ]{hyperref}
 \usepackage{latexsym,amssymb,mathrsfs,fancyhdr}
\font\tenmsb=msbm10 \font\sevenmsb=msbm7 \font\fivemsb=msbm5
\catcode`\@=11 \ifx\amstexloaded@\relax\catcode`\@=\active
\endinput\else\let\amstexloaded@\relax\fi
\def\spaces@{\space\space\space\space\space}
\def\spaces@@{\spaces@\spaces@\spaces@\spaces@\spaces@}
\def\relaxnext@{\let\next\relax}
\def\accentfam@{7}

\def\noaccents@{\def\accentfam@{0}}
\def\mathcal{\relaxnext@\ifmmode\let\next\mathcal@\else
\def\next{\Err@{Use
\string\mathcal\space only in math mode}}\fi\next}
\def\mathcal@#1{{\mathcal@@{#1}}} \def\mathcal@@#1{\noaccents@\fam\tw@#1}

\def\1{{\frac12}}

    \def\eqref #1{{(\ref{#1})}}
\def\Bbb{\relaxnext@\ifmmode\let\next\Bbb@\else
\def\next{\Err@{Use \string\Bbb\space only in math mode}}\fi\next}

\def\Bbb@#1{{\Bbb@@{#1}}}
\def\Bbb@@#1{\noaccents@\fam\msbfam#1}
\newfam\msbfam \textfont\msbfam=\tenmsb \scriptfont\msbfam=\sevenmsb
\scriptscriptfont\msbfam=\fivemsb
\def\N{{\mathbb N}} \def\Z{{\mathbb Z}}   \def\ii{{\rm i}}
\def\R{{\mathbb R}} \def\T{{\mathbb T}}  
\def\C{{\mathbb C}}

\newtheorem{Theorem}{Theorem}
\newtheorem{Lemma}{Lemma}[section]

\newtheorem{Corollary}{Corollary}
\newtheorem{Remark}{Remark}[section]

\@addtoreset{equation}{section}
\newcommand{\sss}{\smallskip}

\newcommand{\bs}{\bigskip}

\newcommand{\qed}{\nolinebreak\hfill\rule{2mm}{2mm}
\par\medbreak}

\newcommand{\la}{\langle } \newcommand{\ra}{\rangle }

\newcommand{\kth}{e^{{\rm i}\langle k, \theta\rangle}\; }

\newcommand{\beq}{\begin{equation} } \newcommand{\eeq}{\end{equation} }

\begin{document}
\setlength{\columnsep}{5pt}
\title{KAM Theorem for a Hamiltonian system with Sublinear Growth Frequencies at Infinity
 }

\author{ Xindong Xu\\
 {\footnotesize School of Mathematics, Southeast University, Nanjing 210089, P.R.China}\\
{\footnotesize Email: xindong.xu@seu.edu.cn}}

\date{}
\maketitle
\begin{abstract}
We prove an infinite-dimensional KAM theorem for a Hamiltonian system with sublinear growth frequencies at infinity. As an application, we prove the reducibility of the linear fractional Schr\"odinger equation with quasi-periodic time-dependent forcing.
\end{abstract}



\section{Introduction }
\sss 
We study a Hamiltonian system with frequencies that grow sublinearly at infinity. That is, we consider
\begin{equation}\label{ham1} H= N+P=\sum\limits_{1\leq j\leq d}
\omega_j (\xi)I_j+\sum\limits_{n\in \Z}\Omega_n(\xi)|z_n|^2
+P(\xi,I,\theta,z,\bar z), \end{equation}
where  $\Omega_n=|n|^\alpha+\lambda+\tilde\Omega_n $ with $0<\alpha< 1$, $\lambda>0$. For such a Hamiltonian,   the gaps between the frequencies are decreasing. For example, let $\alpha={1\over2}$ and $\tilde\Omega_n=0$; one then has $\Omega_n=|n|^{1\over2}+\lambda,n\in\Z$. For any fixed $a\in\Z$,  one has
$\lim\limits_{n\rightarrow\infty }\Omega_{n+a}-\Omega_{n}=0$. This is considerably different from the suplinear growth cases; that is,
\begin{equation}
   \Omega_n= |n|^\alpha+o(|n|^{\alpha-1}),\quad
   \left\{
   \begin{array}{ll}
  \alpha \geq 2,& d>1\\
  \alpha\geq 1,&d=1
   \end{array}\right..
   \end{equation}
We refer the reader to \cite{Be2,B3,B4,CY,CW1,EGK,EK,GT,GXY,GY3,KLiang,K,K1,P3,PX,W,Yuan3} for more information. Nevertheless, there are few results on the Kolmogorov-Arnold-Moser (KAM) theorem for the Hamiltonian \eqref{ham1}.

The study of the above Hamiltonian \eqref{ham1} is motivated by Zakharov \cite{Zak} and  Craig-Sulem \cite{CS1}, who introduced the Hamiltonian structure of a water wave in a channel of infinite depth. We note that there have been many important studies on water waves. The time-periodic and time-quasi-periodic solutions and the standing gravity water waves were derived. For details, see \cite{BBHM}.

In regard to Hamiltonian system with sublinear growth of frequencies, Craig and Worfolk \cite{CW2} presented a Birkhoff normal form. Craig and Sulem \cite{CS} studied the function-space mapping properties of Birkhoff-normal-form transformations of the Hamilton for the equations for water waves. Wu-Xu \cite{WX} gave an infinite-dimensional KAM theorem for the Hamiltonian \eqref{ham1}. In their work, the perturbation maintains conservation of momentum and a strong regularity condition, $X_P:\mathscr P_\C^{\rho,p}\rightarrow\mathscr P_\C^{\rho,\bar p}$ with $\bar p>p$ (see Section \ref{notation and main result} for the definition of space). Following this work, Xu \cite{Xu} relaxed the regularity of the perturbation, that is $X_P:\mathscr P_\C^{\rho,p}\rightarrow\mathscr P_\C^{\rho, p}$. Using the property of T\"oplitz--Lipschity by \cite{EK}, they presented a new KAM theorem. However, the condition concerning momentum conservation is necessary in both results. There are significant differences if the perturbation does not satisfy momentum conservation. Recently, Baldi  et al. \cite{BBHM} obtained a result for  \eqref{ham1} without such restriction. They developed a regularization procedure performance on the linearized PDE at each approximate quasi-periodic solution. Moreover, they used their theory to study the time-quasi-periodic solutions for finite-depth gravity water waves. We also mention the interesting work by Duclos  et al. \cite{POP}, in which the energy growth of similar Hamiltonians is given.

Motivated by \cite{BBHM}, we aim to prove an infinite-dimensional KAM Theorem for the Hamiltonian \eqref{ham1}. Our method is different from that in \cite{BBHM}. We emphasize that the perturbation does not satisfy momentum conservation. Nevertheless, following the idea by \cite{GT}, the regularity condition $X_P:\mathscr P_\C^{\rho,p}\rightarrow\mathscr P_\C^{\rho, \bar p}$ $(\bar p>p)$ is replaced by Assumption $\mathcal B2$(see Section \ref{notation and main result} for definition). As a simple application, this theorem is applied to the reducibility of the fractional nonlinear Schr\"odinger  equation \eqref{beam1}. We believe our method helps in understanding the dynamics of such Hamiltonian systems. The general strategy in proofing the KAM Theorem \ref{KAM} is explained below. Generally, the existence of multiple normal frequencies leads to a complex normal form. To show the main idea, we only consider \eqref{ham1} for simplicity.

As usual, let $R$ (see \eqref{truncation}) be the truncation of $P$. The smallness of $P-R$ is obvious if we reduce the weight $\rho$ (see \eqref{norm3}) a little. We then need to solve the so called homological equation. Concerning the solution of the homological equation, the estimations of $F_{k,n,m}^{11}$ $(k\neq 0)$ are standard if we have diophantine condition
$$|\langle k,\omega\rangle+\Omega\cdot\ell|
\geq { \gamma\over K^{4\tau}},\quad 0<|k|\leq K,$$
where $\Omega=(\cdots,\Omega_n,\cdots)_{n\in\Z}$, $\ell\in\Z^{\N}$, and $|\ell|\leq 2$. However, there is a great difference if $k=0$. For the {\it suplinear growth}($\alpha\geq 1$) cases, as an example, we set $\Omega_n=|n|^2+\tilde\Omega_n (n\in\Z)$, then $|\Omega_{n}-\Omega_m|\geq 1/2$ if $|n|\neq |m|$. Following this computation, the regularity of the vector field $X_F$ is obvious. However, this fact is not appropriate for the {\it sublinear growth} ($0<\alpha<1$) cases we consider. If the vector field $X_P$ only satisfies condition $X_P:\mathscr P_\C^{\rho,p}\rightarrow\mathscr P_\C^{\rho,p}$, we have
$$|F^{11}_{0,n,m}|=|{R^{11}_{0,n,m}\over \Omega_{n}-\Omega_{m}}|
\approx\varepsilon e^{-|n-m|\rho} |n|^{1-\alpha}.$$
We obtain an unbounded vector field $X_F$. That is, there is a strong loss of regularity in the KAM scheme. This is very similar to the claim in \cite{BBHM}, {\it the presence of a sublinear $(\alpha< 1)$ growth of the linear frequencies produce strong losses of derivatives in the iterative KAM scheme}.
To overcome this problem and motivated by \cite{GT}, we assume additionally that $P$ satisfies Assumption $\mathcal B2$ (see Section \ref{notation and main result} for a definition). We then have
$$|P^{11}_{0,n,m}|\leq {\varepsilon e^{-|n-m|\rho}\over \la n\ra^\beta\la m\ra^\beta}.$$
With the restriction $|n-m|\leq K$ and condition $\alpha+\beta\geq1$, we have
\begin{equation}\label{F011}|F^{11}_{0,n,m}|
\leq {\varepsilon e^{-|n-m|\rho}\over \la n\ra^\beta\la m\ra^\beta||n|^\alpha-|m|^\alpha|}
\leq{ \varepsilon e^{-|n-m|\rho}|m|^{1-\alpha}\over |n|^\beta|m|^\beta}
\leq{ \varepsilon e^{-|n-m|\rho}\over |n|^\beta}.\end{equation}
Thus $X_F$ is a regular vector field from $\mathscr P_\C^{\rho,p}$ into itself.

At the same time, using the estimate \eqref{F011} on $F_{0,n,m}^{11}$, one can observe that the homological solution $F$ does not satisfy Assumption $\mathcal B2$. However, we prove that $\{P,F\}$ still satisfies Assumption $\mathcal B2$ and that the new perturbation $P_+$ also satisfies Assumption $\mathcal B2$.

Finally, we introduce a method for estimating the measure of the excluded parameters. We first identify a parameter set $\mathcal O$ with positive measure, such that for any $\xi\in\mathcal O$ and $k\in\Z^d$ with $0<|k|\leq K$, one has
$$|\la k,\omega\ra|\geq {\gamma\over K^\tau}.$$
We next focus mainly on the sets of resonances like
$$\bigcup_{0<|k|\le K,\atop |n-m|\leq K }
\{\xi\in\mathcal O:
|\langle k,\omega\rangle\pm(\Omega_{n}-\Omega_{m})|
\leq {\gamma\over K^\tau}\}.$$
Following an easy computation, one has $||n|^\alpha-|m|^{\alpha}|\leq {\gamma \over 4 |K|^{\tau}}$ if $|n|\geq {K^{2\tau}\over\gamma}$ and $|n-m|\leq K$. Recalling the drift of frequencies, one has $|\tilde\Omega_n|\leq {\varepsilon\over |n|^{2\beta}}(n\in\Z)$. Therefore, for any $\xi\in\mathcal O$, if $|n|\geq {K^{2\tau}\over\gamma}$ and $0<|k|\leq K$, one obtains
$$ |\langle k,\omega\rangle\pm(\Omega_{n}-\Omega_{m})|\geq|\langle k,\omega\rangle|-||n|^\alpha-|m|^\alpha|-|\tilde\Omega_n|-|\tilde\Omega_m| \geq {\gamma\over2 K^\tau}.$$
Therefore, we only need to consider the resonance sets restricted by $0<|k|\leq K$, $|n|\leq{K^{2\tau}\over\gamma}$ and $|n-m|\leq K$. We then prove that the measure of the excluded parameters is bounded by $\gamma$ in the standard way.

\section{An Infinite-Dimensional KAM Theorem }\label{notation and main result}
Let $\mathcal O$ be a positive-measure parameter set in $\R^d$. We consider small perturbations of an infinite-dimensional Hamiltonian in the parameter-dependent normal form
\begin{equation}\label{hamH}
N=\la\omega(\xi),I\ra+\sum\limits_{n\in\Z}\Omega_j z_n\bar z_n
\end{equation}
on phase space
\[\mathscr P^{\rho,p}=\T^d\times \R^d\times \ell^{\rho,p}\times\ell^{\rho,p}\] with coordinate $(\theta,I,z,\bar z)$,
where $\xi\in\mathcal O$, $\T^d=\R^d/2\pi\Z^d$ and $\ell^{\rho,p}$ is the Hilbert space of all real (later complex) sequences $w=(\cdots, w_n,\cdots)_{n\in \Z}$ with norm
\beq\label{norm3}\|w\|_{\rho,p}^2 =\sum_{{n\in\Z}}|w_n|^2e^{2\rho|n|}|n|^{2p},\quad p> 0, \rho>0.\eeq
The complexification of $ \mathscr P^{\rho,p}$ is denoted by $ \mathscr P^{\rho,p}_{\C}$.
The symplectic structure is $dI\wedge d\theta+i\sum_{n\in\Z} dz_n\wedge d\bar z_n.$

The perturbation term $P=P(I,\theta,z,\bar z;\xi)$ is real analytic in $I,\theta,z,\bar z$ and Lipschitz in the parameters $\xi$. For each $\xi\in\mathcal O$, its Hamiltonian vector field $X_P=(-P_{\theta},P_I,iP_z,-iP_{\bar z})$ defines a real analytic map from $\mathscr P^{a,\rho}$ into itself near $\mathcal T_0^{d}= {\T^{d}}\times \{0,0,0\}$. To make this quantitative, we introduce the complex $\mathcal T_0^{d}$-neighborhoods
\begin{equation}D(s,r)=\{({\theta}, I,z,\bar z):|{\rm Im} \theta|<s, |I|<r^2, {|z|}_{\rho,p}<r,
{|\bar {z}|}_{\rho,p}<r\},\end{equation}
where $|\cdot|$ denotes the sup-norm of complex vectors, and the weighted phase space norms are defined as
\begin{equation}\label{norm}
|W|_{r,\rho}=:|W|_{r,\rho,p}=|X|+{1\over r^2}|Y|
+{1\over r}|U|_{\rho,p}+{1\over r}|V|_{\rho,p}\end{equation}
for $W=(X,Y,U,V)$.

Let $P$ be real analytic in $D_{\rho}(s,r)$ for some $s,r>0$ and Lipschitz in $\mathcal O$. We then define the norms $$\|P\|_{D_{\rho}(s,r)}
=\sup_{D_{\rho}(s,r)\times\mathcal O}|P|<\infty$$
and
$$\|P\|^{\mathcal L}_{D(s,r)}
=\sup_{\xi,\eta\in\mathcal O,\atop\xi\neq\eta}
\sum_{D(s,r)}{|\triangle_{\xi\eta}P|\over|\xi-\eta|}<\infty$$
where $\triangle_{\xi\eta}P=P(\cdot,\xi)-P(\cdot,\eta)$. We also define the semi-norms
$$\|X_P\|_{r,D(s,r)}
=\sup\limits_{D(s,r)\times\mathcal O}\|X_P\|_{r,\rho}$$ and
$$\|X_P\|_{r,D(s,r)}^{\mathcal L}
=\sup\limits_{D(s,r)\times\mathcal O, \xi
\neq\eta}{|\triangle_{\xi\eta} X_P|_{r,\rho}\over |\xi-\eta|},$$
where $\triangle_{\xi\eta} X_P=X_P(\cdot,\xi)-X_P(\cdot,\eta)$. For simplicity, we usually write $$\|P\|^*_{D(s,r)}=\|P\|_{D(s,r)}+\|P\|^{\mathcal L}_{D(s,r)},$$
$$\|X_P\|_{r,D(s,r)}^*=\|X_P\|_{r,D(s,r)}+\|X_P\|_{r,D(s,r)}^{\mathcal L}.$$
In the sequel, the semi-norm of any function $f(\xi)$ on $\xi\in\mathcal O$ is defined as
$$|f|_{\mathcal O}^{*}=|f|_{\mathcal O}
+|f|^{\mathcal L}_{\mathcal O},$$
where the Lipschitz semi-norm is defined analogously to $\|X_P\|_{s,D(r,s)}^{\mathcal L}$.

Consider now the perturbed Hamiltonian
\begin{equation}\label{equ-h}
H=\sum\limits_{1\leq j\leq d}\omega_j(\xi) I_j
+\sum\limits_{n\in \Z}\Omega_nz_n\bar z_n+P(I,\theta,z,\bar z;\xi).
\end{equation}
The assumptions imposed on the frequency and the perturbation are given.

\bs \noindent\textbf{Assumption $\mathcal {A}$} (Frequency) \\
\bs \noindent $(A1) ${\it Nondegeneracy :}
The map $\xi\to\omega(\xi)$ is Lipschitz between $\mathcal O$ and its image with $|\omega|^*_{\mathcal O},| \nabla \omega^{-1}|^{\mathcal L}_{\omega(\mathcal O)}\leq M$.

\bs \noindent $(A2)$ {\it Sublinear growth of normal frequencies:}
\begin{equation}\label{asymp1}
\Omega_n=|n|^{\alpha}+\lambda+\tilde\Omega_n(\xi), n\in\Z,
\end{equation}
where $0<\alpha<1$, $\beta>0$, $\lambda>0$ and $\sup\limits_{n\in\Z} ||n|^{2\beta}\tilde\Omega_n|^*_{\mathcal O}\leq L $ with $L\ll 1$ and $LM< 1$.

\begin{Remark}The positive number $\lambda$ is given to avoid some technical problems. \end{Remark}
\bs \noindent\textbf{Assumption $\mathcal {B}$} (Perturbation)\\
\bs \noindent $(\mathcal B1)$ $P$ is real analytic in $I, \theta, z, \bar z$ and Lipschitz in $\xi$; in addition, there exist $r,s>0$ so that $\|X_P\|^*_{\!{}_{r,D_{\rho}(s,r)}}<\infty. $

As in \cite{GT}, we define the space $\Gamma_{r,D_{\rho}(s,r)}^\beta$. We say that $P\in \Gamma_{r,D_{\rho}(s,r)}^\beta$ if $\llbracket P\rrbracket_{r,D_{\rho}(s,r)}^{\beta,*}=:\llbracket P \rrbracket_{r,D_{\rho}(s,r)}^{\beta}+\llbracket P\rrbracket_{r,D_{\rho}(s,r)}^{\beta,\mathcal L}<\infty$.
The norm $\llbracket \cdot \rrbracket_{r,D_{\rho}(s,r)}^\beta$ is defined by the conditions
\begin{eqnarray*}
\|P\|_{D_{\rho}(s,r)}&\leq& r^2 \llbracket P\rrbracket_{r,D_{\rho}(s,r)}^\beta,\\
\max\limits_{1\leq j\leq d}\|\frac{\partial P}{\partial I_j}\|_{D_{\rho}(s,r)}&\leq& \llbracket P\rrbracket_{r,D_{\rho}(s,r)}^\beta,\\
\|\frac{\partial P}{\partial w_n^\iota}\|_{D_{\rho}(s,r)}&\leq& r\llbracket P\rrbracket_{r,D_{\rho}(s,r)}^\beta e^{-|n|{\rho}}\la n\ra^{-\beta},\\
\|\frac{\partial P}{\partial w_n^{\iota_1}\partial w_m^{\iota_2}}\|_{D_{\rho}(s,r)}&\leq& \llbracket P\rrbracket_{r,D_{\rho}(s,r)}^\beta e^{-|\iota_1 n+\iota_2 m|{\rho}}\la n\ra^{-\beta} \la m\ra^{-\beta},
\end{eqnarray*}
where $n,m\in \Z$, $\la n\ra=\max\{{1\over2}, |n|\}$, $\iota=\pm 1$ and $w_n^1=z_n,w_n^{-1}=\bar z_n$. Hence, the semi-norm $\llbracket \cdot \rrbracket_{r,D_{\rho}(s,r)}^{\beta,\mathcal L}$,
\begin{eqnarray*}
\|P\|_{D_{\rho}(s,r)}^{\mathcal L}&\leq& r^2 \llbracket P\rrbracket_{r,D_{\rho}(s,r)}^{\beta,\mathcal L},\\
\max\limits_{1\leq j\leq d}\|\frac{\partial P}{\partial I_j}\|^{\mathcal L}_{D_{\rho}(s,r)}&\leq& \llbracket P\rrbracket_{r,D_{\rho}(s,r)}^{\beta,\mathcal L},\\
\|\frac{\partial P}{\partial w_n^\iota}\|^{\mathcal L}_{D_{\rho}(s,r)}&\leq& r\llbracket P\rrbracket_{r,D_{\rho}(s,r)}^{\beta,\mathcal L} e^{-|n|{\rho}}\la n\ra^{-\beta},\\
\|\frac{\partial P}{\partial w_n^{\iota_1}\partial w_m^{\iota_2}}\|^{\mathcal L}_{D_{\rho}(s,r)}&\leq& \llbracket P\rrbracket_{r,D_{\rho}(s,r)}^{\beta,\mathcal L} e^{-|\iota_1 n+\iota_2 m|{\rho}}\la n\ra^{-\beta} \la m\ra^{-\beta},
\end{eqnarray*}
where $n,m\in \Z$, $\la n\ra=\max\{{1\over2}, |n|\}$, $\iota=\pm 1$ and $w_n^1=z_n,w_n^{-1}=\bar z_n$.

\bs \noindent $(\mathcal B2)$ $P\in \Gamma_{r,D_{\rho}(s,r)}^\beta$ for $\beta>0$.

\bs \noindent Now we are ready to state the infinite-dimensional KAM
theorem.

\begin{Theorem}\label{KAM}
Let $0<\alpha<1$ and $\beta>0$ such that $\alpha+\beta\geq 1$. The Hamiltonian $H=N+P$ is defined on $\mathscr P_\C^{\rho,p}$ for any $\xi\in \mathcal O$. Suppose that the normal form $N$ satisfies Assumption $\mathcal A$ and the perturbation $P$ satisfies Assumption $\mathcal B$ given $s,r,\rho,\gamma>0$. Then there is a positive constant $\varepsilon_0\leq c e^{-{4\rho\over \gamma}}$, such that if $$\|X_P\|^*_{\!{}_{D_{\rho}(s,r)}}+\llbracket P\rrbracket_{r,D_{\rho}(s,r)}^{\beta,*}\leq\varepsilon_0,$$
the following holds:
\begin{itemize}
\item[1)] a Cantor like set $\mathcal O_\gamma$ of $\mathcal O$ with ${\rm meas}(\mathcal O\setminus \mathcal O_\gamma)=O(\gamma^{1\over4})$;
\item[2)] a family of real analytic symplectic maps $\Phi:D_{\rho/2}({s/2},{r/2})\times \mathcal O_\gamma\rightarrow \mathscr{P}_\C^{a,p}$ with
\beq\label{transfromation}\|\Phi-id\|^*_{r/2,D_{\rho/2}({s/ 2},{r/2})}\leq c\varepsilon_0^{1\over2};\eeq
\item[3)] a family of normal forms
$$N^\star+\mathcal A^\star=\la\omega^\star,I\ra+\sum\limits_{j\in\Z} \mathbf\Omega^\star_n(\xi) z_n\bar z_n+\sum\limits_{n\in \Z}a^\star_{n,-n}(\xi)z_n\bar z_{-n}$$
defined on $D_{\rho/2}({s/2},{r/2})\times \mathcal O_ \gamma$, such that
 \begin{equation}
H\circ\Phi=N^\star+\mathcal A^\star+P^\star,
 \end{equation}
where the Taylor series expansion of $P^\star$ only contains monomials of the form $I^m z^q\bar z^{\bar q}$ with $2|m|+|q+\bar q|\geq 3$, and
\beq\label{frequencyshift}
|\omega^\star-\omega|^*_{ \mathcal O_\gamma},\,
\sup\limits_{n\in\Z}||n|^{2\beta}
(\Omega_n^\star-\Omega_n)|^*_{ \mathcal O_\gamma},\,
\sup\limits_{n\in\Z}|n|^{2\beta}
e^{|n|\rho}|a_{n,-n}|^*_{ \mathcal O_\gamma}\leq c\varepsilon_0.\eeq
\end{itemize}
\end{Theorem}
\subsection{Application to the fractional NLS equation}
Imposing periodic boundary conditions, we apply Theorem \ref{KAM} to the fractional NLS equation
\begin{equation}\label{beam1}
iu_t-|\partial_x|^{1\over2}u+\lambda u
=\epsilon \Psi(V(t\omega,x;\xi) \Psi u ),\quad x\in \T,t\in \R,
\end{equation}
where the convolution operator $\Psi:u\rightarrow\Psi*u$ is given with function $\psi(x)$, which is smooth and of order $\beta>0$. More precisely, $\|\Psi u\|_{\rho,p+\beta}\leq c\|u\|_{\rho,p}$.
The parameter $\lambda$ is positive, $\lambda>0$. The function $V:\T^d\times\T\times\mathcal O\ni(\theta,x;\xi)\mapsto\R$ is real analytic in $\theta$ and $x$, and Lipschitz in $\xi$. For $\rho>0$, function $V(\theta,x;\xi)$ extends analytically to the domain $\T^d_\rho\times\T_\rho$, with $\T^d_\rho =\{a+ib\in\C^d/2\pi\Z^d:|b|\leq \rho\}$. It is noted that, in the physics literature, the fractional Schr\"odinger equation was introduced by Laskin \cite{L} in deriving a fractional version of the classical quantum mechanics. Subsequently, many works have been done on such equations; see \cite{FQT,IP} for details.

Let $\{\phi_n(x)=\sqrt{1\over {(2\pi)}}e^{i\langle n,x\rangle} \}_{n\in\Z}$ denote the standard Fourier basis of operator $|\partial_x|^{1\over2}+\lambda$ and $\{\Omega_n=|n|^{1\over2}+\lambda\}_{ n\in\Z}$ be its eigenvalues. Expanding $u$ and $\bar u$ in this basis, specifically, $u=\sum\limits_{n\in \Z} z_n\phi_n(x)$ and $\bar u=\sum\limits_{n\in \Z} \bar z_n\phi_n(x)$, and the equation $\eqref{beam1}$ can be written as a non-autonomous Hamiltonian system
\begin{eqnarray}\label{hs00}
\left\{\begin{array}{l}
\dot{z}_n=-i{\Omega}_n z_n-i\partial_{\bar{z}_n}
P(\omega t,\varphi,z,\bar z;\xi),\ n\in\Z,\\
\dot{\bar{z}}_n=i {\Omega}_n\bar{z}_n+i\partial_{{z}_n}
P(\omega t,\varphi,z,\bar z;\xi),n\in\Z.
\end{array}\right.
\end{eqnarray}
We then re-interpret \eqref{hs00} as an autonomous Hamiltonian system in the extended phase space $\mathscr P^{a,\rho}$,
\begin{eqnarray}\label{ham3}
\left\{\begin{array}{ll}
\dot{I}=-\partial_{{\theta}}P(\theta,z,\bar z;\xi),\\
\dot\theta=\omega,\\
\dot{z}_n=-i{\Omega}_n z_n-i\epsilon\partial_{\bar{z}_n}P(\theta,z,\bar z;\xi),\ n\in\Z,\\
\dot{\bar{z}}_n=i {\Omega}_n\bar{z}_n+i\epsilon\partial_{{z}_n}P(\theta,z,\bar z;\xi),n\in\Z,
\end{array}\right.
\end{eqnarray}
with perturbation
$$P(\theta,z,\bar z;\xi)=\epsilon\int_{\T} V(\theta,x;\xi)( \sum\limits_{n\in \Z} z_n\phi_n(x))(\sum\limits_{n\in \Z} \bar z_n\bar \phi_n(x))dx.$$
The last three equations of \eqref{ham3} are independent of $I$ and are equivalent to $\eqref{beam1}$. Furthermore, \eqref{ham3} determines a Hamiltonian system associated with Hamiltonian
\begin{equation}\label{hamam}H=N+P=\la\omega,I\ra+\sum_{n\in\Z,} {\Omega}_n |z_n|^2+P(\theta,z,\bar z;\xi)\end{equation}
with symplectic structure $dI\wedge d\theta+\ii \sum\limits_{n\in \Z} d z_n\wedge d \bar z_n$. The external parameters are explicitly the frequencies $\omega\in[0,2\pi]^d$.

We now verify that \eqref{hamam} satisfies all the assumptions of Theorem \ref{KAM}.

\noindent {\it Verification of Assumption $\mathcal A$:} It is obvious.\\
\noindent {\it Verification of Assumption $\mathcal B$:}
Since $V(\theta,x;\xi)$ is analytic in $x$ and $\theta$, for any $n\in\Z$, one has uniformly
$$|{\partial P\over\partial{z_n}}|
=|\epsilon\int_{\T} \Psi(V(\theta,x;\xi)\Psi \bar z)\phi_n(x)dx|
\leq c\epsilon r{e^{-|n|\rho}\la n\ra^{-\beta}},\,
\forall\theta\in\T^d,\,\forall\xi\in\mathcal O.$$
Similarly, $|{\partial P\over\partial{\bar z_n}}|\leq c\epsilon r{e^{-|n|\rho}\la n\ra^{-\beta}},\,\forall\theta\in\T^d,\, \forall\xi\in\mathcal O,\forall n\in\Z$. Clearly,
$${\partial^2 P\over \partial{z_m}\partial{ z_n}}
={\partial^2 P\over \partial{\bar z_m}\partial{\bar z_n}}=0,
\forall n,m\in\Z.$$
If we write $\hat V(\theta,x;\xi)=\sum_{k\in\Z}\hat V_k(\theta;\xi)e^{i\la k,x\ra}dx$, then $$|{\partial^2 P\over \partial{z_m}\partial{\bar z_n}}|=|{\epsilon\hat V_{m-n}\over \la n\ra^{\beta}\la m\ra^{\beta}}|\leq {c\epsilon e^{-|n-m|\rho}\over\la n\ra^{\beta}\la m\ra^\beta},\,\forall\theta\in\T_\rho^d,\,\forall\xi\in\mathcal O,\forall n,m\in\Z.$$
Thus, the Assumption $\mathcal B$ obtains if we set $\epsilon$ sufficiently small.

\noindent Following \cite{EK1}, we have
\begin{Theorem}
For any $0 < \epsilon \leq\varepsilon_0$, where $\epsilon_0$ is sufficiently small, there exists a Cantor-like set $\mathcal O_\epsilon$ of positive measure and ${\rm meas}(\mathcal O_\epsilon)\rightarrow (2\pi)^d$ as $\epsilon\rightarrow 0$, such that for $\omega\in\mathcal O_\epsilon$ and $\varphi\in\T^d$, there exists  a complex-linear isomorphism $\Psi=\Psi(\varphi;\omega)$ in the space $L^2(\T^d )$, which depends analytically on $\varphi\in\T_{\rho/2}^d$ and a bounded Hermitian matrix $ A_{\Z\times \Z}$ with
\beq \label{operator}A_{n,m}= 0,\quad \, n\neq-m.\eeq
The following holds: a curve $v(t) = v(t,\cdot )\in L^2(\T^d )$ satisfies the autonomous equation
\beq\label{redu-equation}\dot v = i\triangle v+i\epsilon A v \eeq
if and only if $u(t, \cdot) =\Psi(\varphi_0+t\omega)v(t, \cdot)$ is a solution of \eqref{beam1}.
\end{Theorem}
As $A$ is Hermitian and satisfies \eqref{operator}, then the spectrum of the linear operator on the r.h.s. of \eqref{redu-equation} is a pure point and is imaginary. Hence, all the solutions $v(t)\in L^2(\T^d )$ of
\eqref{redu-equation} are almost-periodic functions of $t$.

\section{KAM STEP}
Theorem \ref{KAM} is proved by a KAM iteration, which involves an infinite sequence of changes in variables. Each step of the KAM iteration makes the perturbation smaller than before in a narrower parameter set and analytic domain. The main task to show is that the new perturbation still satisfies the Assumption $\mathcal B2$.

At the $\nu$--step of the KAM iteration, we consider a Hamiltonian
$$ H_\nu=N_\nu+\mathcal A_\nu+P_\nu$$
defined on $D{\rho_{\nu}}(r_\nu, s_\nu)\times \mathcal O_{\nu}$, where the Assumption $\mathcal A$ and the $\mathcal B$ are satisfied. We construct a symplectic change of variables
$$\Phi_{\nu}: D{\rho_{\nu+1}} (r_{\nu+1}, s_{\nu+1})
\times\mathcal O_{\nu+1} \to D{\rho_{\nu}}(r_\nu, s_\nu),$$
such that the vector field $X_{H_\nu \circ\Phi_\nu}$ defined on $D{\rho_{\nu+1}}(r_{\nu+1}, s_{\nu+1})\times\mathcal O_{\nu+1}$ and satisfies
\[\|X_{P_{\nu+1}}\|^*_{r_{\nu+1},D_{\rho_{\nu+1}}(r_{\nu+1}, s_{\nu+1})}=\|X_{H_\nu \circ\Phi_\nu}-X_{N_{\nu+1}+\mathcal A_{\nu+1}}\|^*_{r_{\nu+1},D{\rho_{\nu+1}}(r_{\nu+1}, s_{\nu+1})}\le\varepsilon_{\nu+1},\]
with new normal form $N_{\nu+1}+\mathcal A_{\nu+1}$. Moreover, the new perturbation $P_{\nu+1}$ still satisfies the  Assumption $\mathcal B$.

\sss For simplicity of notation in the following, the quantities without subscripts refer to the quantities at the $\nu^{\rm th}$ step, whereas the quantities with subscripts $+$ denote the corresponding quantities
at the $(\nu+1)^{\rm th}$ step.

Let us then consider Hamiltonian $H=N+\mathcal A+P$ with
\begin{equation}\label{normalform1}
N=\langle \omega(\xi),I\rangle+\sum_{n\in\Z}\Omega_n(\xi)z_n\bar z_n,\,
\mathcal A=\sum_{1\leq |n|\leq {K}}a_{n,-n}(\xi)z_n\bar z_{-n},\end{equation}
on $D_{\rho}(s,r)\times \mathcal O $, where $|a_{n,-n}|^*_{\mathcal O}\leq \varepsilon_0e^{-2|n|\rho} \la n\ra^{-2\beta}$ and $K$ is the truncation parameter. The corresponding symplectic structure is $dI\wedge d\theta+i\sum\limits_{n\in\Z}dz_n\wedge d\bar z_n$. The normal frequencies are assumed to satisfy
\begin{equation} ||n|^{2\beta}(\Omega_n-|n|^\alpha-\lambda)|^*_{\mathcal O_{\nu}}
\le\varepsilon,\,\forall n\in\Z.\end{equation}
For ease of notation, we set
\begin{equation}\label{normalform}
a_{n,-n}=0 \,\,if\,\, |n|>K
\end{equation}
and define
\beq\label{normalform3} A_0=\Omega_0, \quad A_n= \left(\begin{array}{cc}\Omega_n&a_n \\
a_{-n,n} & \Omega_{-n}\end{array}\right), |n|\geq 1.
\eeq

Let \begin{equation}\label{tau}
\tau=12\tau_1+16\varsigma,\quad\tau_1>d+3+{4\over\alpha^2},
\varsigma={\tau_1+1\over 1-\alpha}.\end{equation}
The parameter $\tau_1$ is only used in the section on the measure estimate. We now assume that, for $\xi\in \mathcal O$ and $|k|\le K$, there is
\begin{eqnarray*}
&&\|\langle k,\omega\rangle^{-1}\|
<{K^{\tau}\over \gamma},\,k\neq 0\nonumber\\
&&\|(\langle k,\omega\rangle {\mathbb I}_n+A_n)^{-1}\|
<{ K^{2\tau}\over\gamma}, \\
&&\|(\langle k,\omega\rangle {\mathbb I}_{nm} \pm( A_{n}
\otimes {\mathbb I}_n+{\mathbb I}_m\otimes A_{m}))^{-1}\|
< { K^{4\tau}\over\gamma},\nonumber\\
&&\|(\langle k,\omega\rangle {\mathbb I}_{nm} \pm( A_{n}
\otimes {\mathbb I}_n-{\mathbb I}_m\otimes A_{m}))^{-1}\|
< {K^{4\tau}\over\gamma },\,k\neq0\&|n-m|<K.\nonumber
\end{eqnarray*}\footnote{The tensor product (or direct product) of two $m\times n$, $k\times l$ matrices $A=(a_{ij}), B$ is $(mk)\times(nl)$ matrix defined by
$$A\otimes B=(a_{ij}B)
= \left(\begin{array}{ccc}a_{11}B&\cdots&a_{1n}B \\
\cdots&\cdots&\cdots\\
a_{nl}B & \cdots &a_{nm}B\end{array}\right).$$
Let $a_{n,-n}=0$, $|n|\neq |m|$, and $n,m\neq 0$, then $diam {\mathbb I}_n=2$ and $diam {\mathbb I}_{nm}=2$
$$\langle k,\omega\rangle {\mathbb I}_{nm}
+( A_{n}\otimes {\mathbb I}_n-{\mathbb I}_m\otimes A_{m})
=diag(\langle k,\omega\rangle+\Omega_{\pm n}-\Omega_{\pm m}).$$
One may refer to \cite{CY} for more information on this symbol.}
where ${\mathbb I}_n$ and ${\mathbb I}_{nm}$ are identity matrices, $diam {\mathbb I}_n=diam A_n$ and $diam{\mathbb I}_{nm}=diam A_n\times diam A_m$.

Let $R$ be the truncation of $P(\theta,I,z,\bar z;\xi)$ with $K$,
 \beq\label{truncation}
R(\theta, I, z,\bar z;\xi)=\sum_{|k|\leq K,2|l|+|q|+|\bar q|
\leq 2,\atop |\sum\limits_{j\in\Z}jq_j+j\bar q_j|\leq K} R_{klq\bar q}(\xi)e^{i(k,\theta)}I^l z^q\bar z^{\bar q},\eeq
where $R_{klq\bar q}=P_{klq\bar q}$. For ease of notation, we rewrite it as
\begin{eqnarray}
R&=&R^0+R^1+R^{10}+R^{01}+R^{20}+R^{11}+R^{02}\nonumber\\
&=&\sum_{|k|\le K} R_{k}^0\kth+\sum_{|k|\le K} \langle
R_{k}^1,I\rangle\kth+\sum_{{|k|\le K,n\in\Z}} R_{k,n}^{10} z_n\kth\nonumber\\
 &&+\sum_{{|k|\le K,n\in\Z}}R^{01}_{k,n}\bar z_n \kth+\sum_{
|k|\le K,\atop|n+m|\leq K}R^{20}_{k,nm}z_n z_m\kth\nonumber\\ &&+\sum_{ |k|\le
K,|n-m|\leq K }R^{11}_{k,nm}z_n\bar z_m\kth+\sum_{ |k|\le
K,\atop |n+m|\leq K}R^{02}_{k,nm}\bar z_n\bar z_m\kth,\nonumber
\end{eqnarray}
where $R_{k,n}^{10}=R_{k0q_n0}$, $q_n=(\cdots,0,\cdots,0,1,0\cdots,0,\cdots)$ and $1$ is at the $n^{\rm th}$ position; $R_{k,n}^{01}=R_{k00q_n}$; $R^{20}_{k,nm}=R_{k0q_{nm}0}$ with $q_{nm}=q_n+q_m$; $R^{11}_{k,nm}=R_{k0q_n\bar q_m}$; $R^{02}_{k,nm}=R_{k00\bar q_{nm}}$ with $\bar q_{nm}=q_n+q_m$.
The generalized mean part of $R$ is defined as
\begin{equation}\label{generalized average part}
\langle R\rangle:= \langle R_{0}^1,I\rangle+\sum_{n\in \Z} R^{11}_{0,nn}|z_n|^2+\sum_{n\in \Z}R^{11}_{0,n,-n}z_n\bar z_{-n}.\end{equation}

Let $F(\theta, I, z,\bar z;\xi)$ be the solution of the so-called homological equation
\begin{equation}\label{homological equation}
\{N+\mathcal A,F\}+R-\langle R\rangle=0.
\end{equation} As  usual, the function $F$ is assumed to have the same form as $R$; that is,
\beq F=F^0+F^1+F^{01}+F^{10}+F^{20}+F^{11}+F^{02}.\eeq
Once we can solve equation \eqref{homological equation} in a proper space, let $X_F^t$ be the flow of $X_F$ at time $t$ associated with the vector field of $F$. We have a new Hamiltonian,
\begin{eqnarray}
H\circ X^1_F
&=&(N+\mathcal A+R)\circ X_F^1+(P-R)\circ X^1_F\label{4.11}\\
&=& N+\{N+\mathcal A,F\}+R+\int_0^1 (1-t)
\{\{N+\mathcal A,F\},F\}\circ X_F^{t}dt\nonumber\\
&&+\int_0^1 \{R,F\}\circ X_F^{t}dt+(P-R)\circ X^1_F\nonumber\\
&=& N_++\mathcal A_++P_+,\nonumber
\end{eqnarray}
where the new perturbation,
\beq\label{newperturbation}P_+=:\int_0^1 \{(1-t)
\langle R\rangle+tR,F\}\circ X_F^{t}dt+(P-R)\circ X^1_F,\eeq
and the new normal forms $N_+$ and $\mathcal A_+$ have the same form as \eqref{normalform1} with
\beq\label{newnormalform2}
\omega_+(\xi)=\omega+R_{0}^1, \,\, \Omega_n^+=\Omega_n+R^{11}_{0,nn},\,\, a_{n,-n}^+= a_{n,-n}+R^{11}_{0,n,-n}.
\eeq

It is easy to check that the function $F$ is not in the space $\Gamma_{r,D_\rho(s,r)}^{\beta}$ because the growth of frequencies is sublinear (see \ref{F11-2}). We shall prove that the homological solution $F$ is in class $\Gamma_{r,D(s,r)}^{\beta,\alpha}$. We say that $F\in \Gamma_{r,D(s,r)}^{\beta,\alpha}$ if $\llbracket F\rrbracket_{r,D(s,r)}^{\beta,\alpha,*}<\infty$. Like $\llbracket \cdot \rrbracket_{r,D(s,r)}^{\beta,*}$, the semi-norm $\llbracket \cdot \rrbracket_{r,D(s,r)}^{\beta,\alpha,*}$ is defined by the conditions
 \begin{eqnarray*}
\|F\|^*_{D_{\rho}(s,r)}\leq r^2 \llbracket F\rrbracket^{\beta,\alpha,*}_{r,D_{\rho}(s,r)},
&&
\max\limits_{1\leq j\leq d}\|\frac{\partial F}{\partial I_j}\|^*_{D_{\rho}(s,r)}\leq \llbracket F\rrbracket^{\beta,\alpha,*}_{r,D_{\rho}(s,r)},\\
\|\frac{\partial F}{\partial z_n}\|^*_{D_{\rho}(s,r)},\|\frac{\partial F}{\partial\bar z_n}\|^*_{D_{\rho}(s,r)}&\leq& r\llbracket F\rrbracket^{\beta,\alpha,*}_{r,D_{\rho}(s,r)} e^{-|n|\rho}\la n\ra^{-\beta},\\
\|\frac{\partial F}{\partial \bar z_n\partial \bar z_m}\|^*_{D_{\rho}(s,r)}, \|\frac{\partial F}{\partial z_n\partial z_m}\|^*_{D_{\rho}(s,r)}&\leq& { \llbracket F\rrbracket^{\beta,\alpha,*}_{r,D_{\rho}(s,r)} e^{-|n+m|\rho}\over\la n\ra^{\beta}\la m\ra^{\beta}},\\
\|\frac{\partial^2 F}{\partial z_n\partial \bar z_n}\|^*_{D_{\rho}(s,r)}&\leq& { \llbracket F\rrbracket^{\beta,\alpha,*}_{r,D_{\rho}(s,r)} {\la n\ra^{-2\beta}}},\\
\|\frac{\partial^2 ( F-[F])}{\partial z_n\partial \bar z_m}\|^*_{D_{\rho}(s,r)}&\leq& { \llbracket F\rrbracket^{\beta,\alpha,*}_{r,D_{\rho}(s,r)} e^{-| n-m |\rho}\over {\la n\ra^{\beta}\la m\ra}^{\beta}},\quad |n|\neq|m|,\\
\|\frac{\partial^2 [F]}{\partial z_n\partial \bar z_m}\|^*_{D_{\rho}(s,r)}&\leq& { \llbracket F\rrbracket^{\beta,\alpha,*}_{r,D_{\rho}(s,r)} e^{-|n-m|\rho}\over {\la n\ra^{\beta}\la m\ra^{\beta}}||n|^\alpha-|m|^\alpha|},\, \, |n|\neq|m|,
 \end{eqnarray*}
where $[F(\theta,I,z,\bar z;\xi)]=\int_{\T}F(\theta,I,z,\bar z;\xi)d\theta$ and $n,m\in\Z$.

\subsection{Homological Equation}
We next solve the homological equation and then prove that $F\in\Gamma_{r,D_{\rho}(s-\sigma,r)}^{\beta,\alpha}$. The regularity of $F$ is also given.
\begin{Lemma} \label{Lem4.2} Let $0<\sigma<s$, $0<\mu<\rho, K>0$, and $R\in\Gamma^{\beta}_{r,D_\rho(s,r)}$ of the form
$$R=\sum_{|k|\leq K,2|l|+|q|+|\bar q|\leq 2,\atop |\sum\limits_{j\in\Z}jq_j+j\bar q_j|\leq K} R_{klq\bar q}e^{i(k,\theta)}I^l z^q\bar z^{\bar q}.$$
Assume that for any $\xi\in\mathcal O$, $|k|\leq K$ and $n,m\in\Z$, we have
\begin{eqnarray}
&&\|\langle k,\omega\rangle^{-1}\|< {\gamma\over K^{\tau}},\,k\neq 0\nonumber\\
&&\|(\langle k,\omega\rangle {\mathbb I}_n+A_n)^{-1}\|
<{ K^{2\tau}\over\gamma},\label{smalldivisorassumption} \\
&&\|(\langle k,\omega\rangle {\mathbb I}_{nm} \pm( A_{n}\otimes {\mathbb I}_n+{\mathbb I}_m\otimes A_{m}))^{-1}\|< {K^{4\tau}\over \gamma},\nonumber\\
&&\|(\langle k,\omega\rangle {\mathbb I}_{nm} \pm( A_{n}\otimes {\mathbb I}_n-{\mathbb I}_m\otimes A_{m}))^{-1}\|
< {K^{4\tau}\over \gamma},\,k\neq0\&|n-m|<K.\nonumber
\end{eqnarray}
Then the homological equation \eqref{homological equation} has a solution $F(\theta,I,z,\bar z;\xi)$ with $F\in\Gamma_{r,D_{\rho}(s-\sigma,r)}^{\beta,\alpha}$, such that
\begin{equation}\label{Homological solution}\llbracket F\rrbracket_{r,D_{\rho}(s-\sigma,r)}^{\beta,\alpha,*}
\leq {CK^{8\tau}\llbracket R\rrbracket_{r,D_\rho(s,r)}^{\beta,*} \over \gamma^2\sigma^{d+1}}.\end{equation}
\end{Lemma}

\begin{proof}
From the structure of $N$ and $R$, the homological equation \eqref{homological equation} is equivalent to
\begin{equation}\label{4.131}
\{N,F^0+F^1\}+R^0+R^1-\la R_0^1, I\ra=0,
\end{equation}
\begin{equation}\label{4.132}
\{N+\mathcal A,F^{10}\}+R^{10}=0,
\end{equation}
\begin{equation}\label{4.1322}
\{N+\mathcal A,F^{01}\}+R^{01}=0,
\end{equation}
\begin{equation}\label{4.133}
\{N+\mathcal A,F^{11}\}+R^{11}-\sum_{|n|=|m|}R^{11}_{0,nm}z_n\bar z_m=0,
\end{equation}
\begin{equation}\label{4.134}
\{N+\mathcal A,F^{20}\}+R^{20}=0,
\end{equation}
\begin{equation}\label{4.1344}
\{N+\mathcal A,F^{02}\}+R^{02}=0.
\end{equation}
\noindent$\clubsuit$ Solving the homological equation.

\noindent
{\it Solving (\ref{4.131}):} Let $j=0$ or $1$, then $F^j(\theta)= \sum\limits_{0<|k|\le K} F^j_k\kth$ are constructed by setting
$$F_k^j=\frac{1}{i\langle k,\omega\rangle}R_k^j,0<|k|\le K, j=0,1.$$
Given the assumption \eqref{smalldivisorassumption}, for $0<|k|\le K$ and $\xi\in \mathcal O$, there is
$$\|\langle k,\omega(\xi)\rangle^{-1}\|< {K^{\tau}\over\gamma }.$$
Since $R\in\Gamma^{\beta}_{r,D_\rho(s,r)}$, we have
\beq\label{F0F1}|F_k^j|_{\mathcal O}
\leq r^{2-2j}\gamma^{-2} K^{2\tau}\llbracket R\rrbracket_{r,D_{\rho}(s,r)}^{\beta,*} ,0<|k|\le K,j=0,1.\eeq

\noindent{\it Solving (\ref{4.132}):}
For any $n\in\Z$, we have
\begin{eqnarray}\label{M2}
&&(\langle k,\omega\rangle+\Omega_n)F^{10}_{k,n}+a_{n,-n}F^{10}_{k,-n}
=-{\rm i}R^{10}_{k,n},\\
&&(\langle k,\omega\rangle+\Omega_{-n})F^{10}_{k,-n}+a_{-n,n}F^{10}_{k,n}
=-{\rm i}R^{10}_{k,-n}.\nonumber
\end{eqnarray}

The above equations can be written as

$$(\langle k,\omega\rangle {\mathbb I}_n+A_{n})Q_{k,|n|}^{10}
={-{\rm i}}R_{k,|n|}^{10}$$
with
$$Q_{k,|n|}^{10}=(F_{k,n}^{10},F_{k,-n}^{10}),\,R_{k,|n|}
=(R_{k,n}^{10},R_{k,-n,}^{10}).$$
As $R\in\Gamma^{\beta}_{r,D(s,r)}$, one has $$|R^{10}_{k,n}|_{\mathcal O},|R^{10}_{k,n}|_{\mathcal O}
\leq r\llbracket R\rrbracket_{r,D_{\rho}(s,r)}^{\beta,*} e^{-|k|s}e^{-|n|\rho}\la n\ra^{-\beta}.$$
By the small-divisor assumptions \eqref{smalldivisorassumption},
$$\|(\langle k,\omega\rangle {\mathbb I}_n+A_{n})^{-1}\|
< {K^{2\tau}\over \gamma}, |k|\le K,$$
we obtain
\beq\label{F10F01}|F^{10}_{k,n}|_{\mathcal O}
\leq\gamma^{-2}K^{4\tau}r\llbracket R\rrbracket_{r,D_{\rho}(s,r)}^{\beta,*}
 e^{-|k|s-|n|\rho}\la n\ra^{-\beta}.\eeq
The equation \eqref{4.1322} can be done in the same way.

\noindent{\it Solving (\ref{4.133}):}
First, we consider instances with $k\neq 0$. Comparing the Fourier coefficients, we have $F^{11}_{k, nm}$, $F^{11}_{k, n,-m}$, $F^{11}_{k,-n,m}$, $F^{11}_{k,-n,-m}$ satisfying
$$(\langle k,\omega\rangle+\Omega_n-\Omega_{m})F^{11}_{k,n,m}
+a_{-n,n}F^{11}_{k,-n,m}-a_{m,-m}F^{11}_{k,n,-m}
=-{\rm i}R^{11}_{k,n,m},$$
$$(\langle k,\omega\rangle+\Omega_n
-\Omega_{-m})F^{11}_{k, n,-m}+a_{-n,n}F^{11}_{k,-n,-m}
-a_{-m,m}F^{11}_{k,n,m}
=-{\rm i} R^{11}_{k, n,-m},$$
$$(\langle k,\omega\rangle+\Omega_{-n}
-\Omega_{m}) F^{11}_{k,-n,m}+a_{n,-n}
F^{11}_{k, n,m}-a_{m,-m}F^{11}_{k,-n,-m}=-{\rm i} R^{11}_{k,-n,m},$$
$$(\langle k,\omega\rangle+\Omega_{-n}-\Omega_{-m})F^{11}_{k,-n,-m}
+a_{n,-n}F^{11}_{k, n,-m}-a_{-m,m}F^{11}_{k,-n,m}
=-{\rm i}R^{11}_{k,-n,-m}.$$
These equations can be written as
\beq\label{SolveF11}
(\langle k,\omega\rangle {\mathbb I}_{nm}+A_{n}\otimes {\mathbb I}_n
-{\mathbb I}_m\otimes A_{m}) Q_{k,|n|,|m|}^{11}
={-{\rm i}}R_{k,|n|,|m|}^{11}\eeq
with
$$Q_{k,|n|,|m|}^{11}=(F_{k,n,m}^{11},F_{k,n,-m}^{11},
F_{k,-n,m}^{11},F_{k,-n,-m}^{11}),$$
$$R_{k,|n|,|m|}^{11}=(R_{k,n,m}^{11},R_{k,n,-m}^{11},
R_{k,-n,m}^{11},R_{k,-n,-m}^{11}).$$

As $R\in\Gamma^{\beta}_{r,D(s,r)}$, one has $$|R_{k,n,m}^{11}|_{\mathcal O}
\leq\llbracket R\rrbracket_{r,D_{\rho}(s,r)}^{\beta,*} e^{-|k|s}
e^{-|n-m|\rho}\la n\ra^{-\beta}\la m\ra^{-\beta},
\forall n,m\in\Z.$$
Thus with the small divisor assumption \eqref{smalldivisorassumption},
$$\|(\langle k,\omega\rangle {\mathbb I}_{nm}
\pm (A_{n}\otimes {\mathbb I}_n
-{\mathbb I}_m\otimes A_{m}))^{-1}\|<{ K^{4\tau}\over\gamma}, $$
we have
\begin{equation}\label{F11-1}
|F^{11}_{k,n,m}|_{\mathcal O}\leq \gamma^{-2}K^{8\tau}\llbracket R\rrbracket_{r,D_{\rho}(s,r)}^{\beta,*} e^{-|k|s} e^{-|n-m|\rho}\la n\ra^{-\beta}\la m\ra^{-\beta},k\neq 0.\end{equation}

Second, we solve \eqref{4.133} setting $k=0$. By \eqref{4.133}, we only need to consider instances $|n|\neq |m|$, for which the equation \eqref{SolveF11} takes the form
$$(A_{n}\otimes {\mathbb I}_n-{\mathbb I}_m\otimes A_{m})
Q_{0,|n|,|m|}^{11}={-{\rm i}}R_{0,|n|,|m|}^{11}.$$
Recall \eqref{normalform3}, $\alpha+\beta\geq 1$ and Lemma \ref{numberinequalty}, the matrix $A_{n}\otimes {\mathbb I}_n-{\mathbb I}_m\otimes A_{m}$ is diagonally dominant. One has
$$\|(A_{n}\otimes {\mathbb I}_n-{\mathbb I}_m\otimes A_{m})^{-1}
\|\leq {1\over 2}||n|^\alpha-|m|^{\alpha}|^{-1}$$
and then 
\beq\label{F11-2}|F^{11}_{0,n,m}|_{\mathcal O}\leq \llbracket R\rrbracket_{r,D_{\rho}(s,r)}^{\beta,*} e^{-|n-m|\rho}\la n\ra^{-\beta}\la m\ra^{-\beta}||n|^\alpha-|m|^{\alpha}|^{-1}.\eeq

\noindent{\it Solving (\ref{4.134}):}
Comparing the Fourier coefficients, we have that $F^{20}_{k, nm}$, $F^{20}_{k, n,-m}$, $F^{20}_{k,-n,m}$, and $F^{20}_{k,-n,-m}$ satisfy
$$(\langle k,\omega\rangle+\Omega_n+\Omega_{m})F^{20}_{k,n,m}
+a_{n,-n}F^{20}_{k,-n,m}+a_{m,-m}F^{11}_{k,n,-m}
=-{\rm i}R^{20}_{k,n,m},$$
$$(\langle k,\omega\rangle+\Omega_n+\Omega_{-m})F^{20}_{k, n,-m}+a_{n,-n}F^{20}_{k,-n,-m}+\bar a_{m,-m}F^{20}_{k,n,m}=-{\rm i}
R^{20}_{k, n,-m},$$
$$(\langle k,\omega\rangle+\Omega_{-n}+\Omega_{m}) F^{20}_{k,-n,m}
+a_{n,-n} F^{20}_{k, n,m}+a_{m,-m}F^{20}_{k,-n,-m}
=-{\rm i} R^{20}_{k,-n,m},$$
$$(\langle k,\omega\rangle+\Omega_{-n}
+\Omega_{-m})F^{20}_{k,-n,-m}+a_{n,-n}F^{20}_{k, n,-m}
+a_{m,-m}F^{20}_{k,-n,m}=-{\rm i}R^{20}_{k,-n,-m}.$$

The above equations can be rewritten as
$$(\langle k,\omega\rangle {\mathbb I}_{nm}
+A_{n}\otimes {\mathbb I}_n+{\mathbb I}_m\otimes A_{m})
Q_{k,|n|,|m|}^{20}={-{\rm i}}R^{20}_{k,|n|,|m|}$$
with
$$Q_{k,|n|,|m|}^{20}=(F_{k,n,m}^{20},F_{k,n,-m}^{20},
F_{k,-n,m}^{20},F_{k,-n,-m}^{20}),$$
$$R_{k,|n|,|m|}^{20}=(R_{k,n,m}^{20},R_{k,n,-m}^{20},
R_{k,-n,m}^{20},R_{k,-n,-m}^{20}).$$

As $R\in\Gamma^{\beta}_{r,D(s,r)}$, one has $$|R_{k,n,m}^{20}|_{\mathcal O}\leq
\llbracket R\rrbracket^{\beta,*}_{r,D_{\rho}(s,r)} e^{-|k|s}e^{-|n+m|\rho}\la n\ra^{-\beta}\la m\ra^{-\beta},
\forall n,m\in\Z.$$
Recalling the small divisor assumption \eqref{smalldivisorassumption}, one has
$$\|(\langle k,\omega\rangle{\mathbb I}_{nm} \pm (A_{n}\otimes {\mathbb I}_n+{\mathbb I}_m\otimes A_{m}))^{-1}\|
< { K^{4\tau}\over\gamma},$$
and then
\beq\label{F20F02}|F^{20}_{k,nm}|_{\mathcal O}\leq\gamma^{-2}K^{8\tau}\llbracket R\rrbracket^{\beta,*}_{r,D_{\rho}(s,r)}{ e^{-|k|s}e^{-|n+m|\rho}\over\la n\ra^{\beta}\la m\ra^{\beta}}.\eeq
The equation \eqref{4.1344} can be treated in the same way.

$\clubsuit$ To complete the proof, it suffices to estimate $\frac{\partial F}{\partial z_n}$ and $\frac{\partial^2F}{\partial z_n\partial \bar z_m}$. Using \eqref{F0F1}, \eqref{F10F01}, \eqref{F11-1}, \eqref{F11-2}, \eqref{F20F02}, and Lemma \ref{numberinequalty}, we take the sum in $m$ and $k$,
\begin{eqnarray}
&&|\frac{\partial F}{\partial z_n}|_{D_{\rho}(s-\sigma,r)}\\
&=&|\sum_{|k|\leq K}F^{10}_{kn}+\sum_{|k|\leq K\atop |n+m|
\leq K}F^{20}_{knm} z_m+\sum_{|k|\leq K,\atop |n-m|<K}F^{11}_{knm}\bar z_m|_{D_{\rho}(s-\sigma,r)}\nonumber\\
&\leq& {K^{8\tau} \llbracket R\rrbracket^{\beta,*}_{r,D_{\rho}(s,r)}\over \gamma^2\sigma^{d+1}\la n\ra^{\beta}}({r e^{-|n|\rho}}+\sum_{|n+m|\leq K}{e^{-|n+m|\rho}\over\la m\ra^{\beta}}{re^{-|m|\rho}\over \la m\ra^{p}}
+\sum_{|n-m|\leq K}{e^{-|n-m|\rho}\over\la m\ra^{\beta}||n|^\alpha-|m|^\alpha|}{re^{-|m|\rho}\over \la m\ra^{p}})\nonumber\\
&\leq& {rK^{8\tau} \llbracket R\rrbracket^{\beta,*}_{r,D_{\rho}(s,r)e^{-|n|\rho}}\over \gamma^2\sigma^{d+1}\la n\ra^{\beta}}(1+\sum_{|n+m|\leq K}{1\over\la m\ra^{\beta} \la m\ra^{p}}+\sum_{|n-m|\leq K}{1\over\la m\ra^{\beta+\alpha-1}\la m\ra^{p}})\nonumber\\
&\leq& {rK^{8\tau+1} \llbracket R\rrbracket^{\beta,*}_{r,D_{\rho}(s,r)}e^{-|n|\rho}\over \gamma^2\sigma^{d+1}\la n\ra^{\beta}}.\nonumber
\end{eqnarray}
The last inequality follows as $\alpha+\beta\geq 1$.

If $|n|\neq |m|$, we have
\begin{eqnarray}\label{secondderivative}
|\frac{\partial^2 (F-[F])}{\partial z_n\bar\partial z_m}|_{D_{\rho}(s-\sigma,r)}
&\leq&|\sum_{0<|k|\leq K}F^{11}_{k,n,m}e^{i\la k,\theta\ra}|_{D_{\rho}(r-\sigma,s)}\label{F11F112}\\
&\leq&{\llbracket R\rrbracket_{r,D(s,r)}^{\beta,*} e^{-|n-m|\rho}{\sum\limits_{0<|k|\leq K}\gamma^{-2}K^{8\tau}e^{-|k|\sigma}}\over \la n\ra^{\beta}\la m\ra^{\beta}}\nonumber\\
&\leq&{K^{8\tau}\llbracket R\rrbracket^{\beta,*}_{r,D_{\rho}(s,r)}\over \gamma^2 \sigma^{d+1} }\cdot{e^{-|n-m|\rho}\over\la n\ra^{\beta}\la m\ra^{\beta}}\nonumber
\end{eqnarray}
and
\begin{eqnarray}\label{F11average}
|\frac{\partial^2[F]}{\partial z_n\partial \bar z_m}|_{D_{\rho}(s-\sigma,r)}\leq|F^{11}_{0,n,m}|_{D(s-\sigma,r)}
\leq{C\llbracket R\rrbracket_{r,D_{\rho}(s,r)}^{\beta,*} \cdot e^{-|n-m|\rho}\over \la n\ra^{\beta}\la m\ra^{\beta}||n|^\alpha-|m|^\alpha|}.
\end{eqnarray}
If $n=m$, then
\begin{eqnarray*}
|\frac{\partial^2 F}{\partial z_n\bar\partial z_n}|_{D_{\rho}(s-\sigma,r)}
&\leq&|\sum_{0<|k|\leq K}F^{11}_{k,nn}
e^{i\la k,\theta\ra}|_{D(s-\sigma,r)}
\leq{\llbracket R\rrbracket^{\beta,*}_{r,D_{\rho}(s,r)}K^{2\tau}\over \gamma^2 \sigma^{d+1} \la n\ra^{2\beta}}.\nonumber
\end{eqnarray*}
If $n=-m$ and $|n|\geq K^2$, one has $\frac{\partial^2 F}{\partial z_n\partial \bar z_{-n}}=0$ by the restriction on \eqref{truncation}.

With these observations, we have
$$\llbracket F\rrbracket_{r,D_{\rho-\mu}(s-\sigma,r)}^{\beta,\alpha}
\leq {CK^{8\tau+1}\llbracket R\rrbracket_{r,D_\rho(s,r)}^{\beta,*} \over\gamma^2\sigma^{d+1}}.$$

$\clubsuit$ The estimation on the Lipschitz semi-norm of $F$ is standard. Here we consider only $F^{11}_{0,n,m}$ as an example.
Recall \eqref{normalform3} and let $|n|,|m|\geq K^3$, \eqref{4.133} can be written as
$$(\Omega_n-\Omega_{m})F^{11}_{0,n,m}=-{\rm i}R^{11}_{0,n,m}.$$
One has
\begin{eqnarray*}
\triangle_{\xi\eta} F^{11}_{0,n,m}
&=&-{{\rm i}\triangle_{\xi\eta} R^{11}_{0,n,m}
+F^{11}_{0,n,m}\triangle_{\xi\eta}(\Omega_n-\Omega_{m})\over \Omega_n-\Omega_{m}}\\
&=&-{{\rm i}\triangle_{\xi\eta} R^{11}_{0,n,m}
+F^{11}_{0,n,m}\triangle_{\xi\eta}(\tilde\Omega_n
-\tilde\Omega_{m})\over \Omega_n-\Omega_{m}}
\end{eqnarray*}
By \eqref{F11-2}, we have
$$|\triangle_{\xi\eta} F^{11}_{0,n,m}|
\leq{\triangle_{\xi\eta} R^{11}_{0,n,m}\over ||n|^\alpha
-|m|^{\alpha}|}+{\llbracket R\rrbracket_{r,D_{\rho}(s,r)}^\beta e^{-|n-m|\rho}\over \la n\ra^{\beta}\la m\ra^{\beta}(|n|^\alpha
-|m|^{\alpha})^2} {|\triangle_{\xi\eta}
(\tilde\Omega_n-\tilde\Omega_{m})|}.$$
Hence
$${|\triangle_{\xi\eta} F^{11}_{0,n,m}|\over|\xi-\eta|}
\leq{{|\triangle_{\xi\eta} R^{11}_{0,n,m}|\over|\xi-\eta|}\over ||n|^\alpha-|m|^{\alpha}|}+{\llbracket R\rrbracket_{r,D_{\rho}(s,r)}^\beta e^{-|n-m|\rho}\over \la n\ra^{\beta}\la m\ra^{\beta}(|n|^\alpha-|m|^{\alpha})^2} {{|\triangle_{\xi\eta}(\tilde\Omega_n-\tilde\Omega_{m})|\over |\xi-\eta|}}.$$
Note that  $|\tilde\Omega_n|^*_{\mathcal O}\leq {L\over |n|^{2\beta}}$ for $n\in\Z$, we have
\beq\label{lipschitz f 11}
|F^{11}_{0,n,m}|_{\mathcal O}^{\mathcal L}
\leq{ \llbracket R\rrbracket_{r,D_{\rho}(s,r)}^{\beta,*} e^{-|n-m|\rho}\over \la n\ra^{\beta}\la m\ra^{\beta}}({1\over ||n|^\alpha-|m|^{\alpha}|}
+{1\over (|n|^\alpha-|m|^{\alpha})^2} {L\over |n|^{2\beta}}).\eeq
Recall that by \eqref{truncation}, one has $|n-m|\leq K$. Let $a=n-m\neq0$; then by Lemma \ref{numberinequalty} and condition $\alpha+\beta\geq 1$, we have
$${1\over (|n|^\alpha-|m|^{\alpha})^2}
{1\over |n|^{2\beta}}\leq{L\over |n|^{2\beta+2\alpha-2}}\leq L. $$
Finally, there is
\beq\label{F11-3}|F^{11}_{0,n,m}|_{\mathcal O}^{\mathcal L}
\leq { C\llbracket R\rrbracket_{r,D_{\rho}(s,r)}^{\beta,*}
e^{-|n-m|\rho}\over \la n\ra^{\beta}
\la m\ra^{\beta}||n|^\alpha-|m|^{\alpha}|}.\eeq
Estimates of the others can be obtained in the same way and thus we immediately have our conclusion .\qed
\end{proof}

\noindent The regularity of $X_F^1$ is given by the following lemma.
\begin{Lemma}\label{homological solution is regular}
Let $\alpha+\beta\geq1$; if $F$ is the homological solution given in Lemma \ref{Lem4.2}, we then have
$$\|X_F\|^*_{D_{\rho-\mu}(s-2\sigma,r)}\leq {CK^{8\tau+1}\over \gamma^2 \mu^{2}\sigma^{d+1} }\|X_R\|^*_{D_{\rho}(s,r)}.$$
\end{Lemma}
\begin{proof}
Following \cite{P3,EK}, by \eqref{F0F1}, \eqref{F10F01}, \eqref{F11-1}, \eqref{F11-2} and \eqref{F20F02},
 the proof of this Lemma is standard once we can have a proper bound on \eqref{F11-2}, that is $|F^{11}_{0,n,m}|^*_{\mathcal O}$.
Recall the restriction on \eqref{truncation}, one has $|n-m|\leq K$. A similar restriction applies to $F$ from \eqref{homological equation}. Setting $a=n-m\neq0$,
then by \eqref{F11-2},\eqref{lipschitz f 11} and Lemma \ref{numberinequalty},
$$|F^{11}_{0,n,m}|^*_{\mathcal O}
\leq\llbracket R\rrbracket^{\beta,*}_{r,D_{\rho}(s,r)} \cdot
{e^{-|n-m|\rho}\la m\ra^{1-\alpha}\over \la n\ra^{\beta}
\la m\ra^{\beta}}\leq {{\llbracket R\rrbracket_{r,D(s,r)}^{\beta,*}} \cdot e^{-|n-m|\rho}\over \la n\ra^{\beta}}.$$
The last inequality is possible because $\alpha+\beta\geq1$. We then have our conclusion.\qed
\end{proof}

\begin{Lemma}\label{Lem4.4}
Let $\eta=\varepsilon^{\frac 13}, D_{i\eta}= D_{\rho-\mu}(s_++\frac
{i}4\sigma,\frac i4 \eta r), 0 <i \le 4$. If $\varepsilon\ll (\frac
12\gamma^2 K^{-8\tau-1})$, we then have
\begin{equation}
X_F^t: D_{2\eta} \to D_{3\eta} ,\ \ \-1 \le t\le 1.
\label{4.26}
\end{equation}
Moreover,
\begin{equation}
\|DX_F^t-Id\|^*_{\eta r,\eta r,D_{4\eta}}\le {CK^{8\tau+1}\varepsilon \over \gamma^2 \mu^{p+1}\sigma^{d+2}}. \label{4.27}
\end{equation}
\end{Lemma}
In the above, following \cite{P3}, we use the operator norm $\pmb{\pmb |}L\pmb{\pmb |}_{r,s}=\sup\limits_{ W\neq0} {|LW|_{r}\over|W|_{s}}$ with $|\cdot|_{r}$ defined in \eqref{norm}

Indeed, following \cite{EK1} and \cite{GT}, we also prove the following:
\begin{Corollary}\label{corollary 1}
The symplectic transformation $X_{F}^1$ reads
\begin{equation}\label{structure}
\left(\begin{array}{c}I\\ \theta\\ Z\end{array}\right)\longmapsto
\left(
\begin{array}{c}
I+M(\theta,Z)+L(\theta)Z\\
\theta\\
T(\theta)+U(\theta)Z
\end{array}\right)
\end{equation}
where $M(\theta,Z)$ is quadratic in $Z$, $L(\theta)$ and $U(\theta)$ are bounded linear operators from $\ell^{a,\rho}\times\ell^{a,\rho}$ in $\R^d$ and $\ell^{a,\rho}\times\ell^{a,\rho}$, respectively.
\end{Corollary}

\subsection{Estimate of the Poisson Bracket}
\begin{Lemma}\label{mainLemma}
Let $\alpha,\beta$ be positive numbers such that $\alpha+\beta\geq1$. If $R\in\Gamma^\beta_{r,D_{\rho}(s,r)}$ and $F$ is the homological solution of \eqref{homological equation}. Then, for any $0<4\sigma<s$, $0<\mu<\rho$ and $n\in\N$, we have the following
\begin{equation}\label{poisson1}\llbracket \{R,F\}\rrbracket^{\beta,*}_{r,D_{\rho-\mu}(s-2\sigma,{r/2})}\leq {C\llbracket F\rrbracket^{\beta,\alpha,*}_{r,D_{\rho}(s-\sigma,r)}\llbracket R\rrbracket^{\beta,*}_{r,D_{\rho}(s,r)}\over\sigma\mu^{p+1}(\rho-\mu)^2},\end{equation}
\begin{equation}\label{poisson2}\llbracket\{\cdots\{R,\underbrace{F\}\cdots
,F}_n\}\rrbracket^{\beta,*}_{r,D_{\rho-\mu}(s-2\sigma,{r/2})}\leq ({C\llbracket F\rrbracket^{\beta,\alpha,*}_{r,D_{\rho}(s-\sigma,r)}\over\sigma\mu^{p+1}(\rho-\mu)^2})^n\llbracket R\rrbracket^{\beta,*}_{r,D_\rho(s,r)}.\end{equation}
\end{Lemma}

\begin{proof}
The estimates  \eqref{poisson1} and \eqref{poisson2} are proved in the same way. We show the first in detail.
The expansion of $\{R,F\}$ reads,
$$\{R,F\}=\sum_{1\leq j\leq d}({\partial R\over\partial \theta_j}{\partial F\over\partial I_j}-{\partial F\over\partial \theta_j}{\partial R\over\partial I_j})
+i\sum_{n\in\Z}({\partial R\over\partial z_n}{\partial F\over\partial \bar z_n}-{\partial F\over\partial z_n}{\partial R\over\partial \bar z_n}).$$
It remains to estimate each term of this expansion and its derivatives.

Note that $F$ is of degree $2$; we have
\beq\label{functionstructure4}{\partial^2 F\over\partial z\partial I}={\partial^2 F\over\partial I^2}={\partial^3F\over\partial w^3}=0, w=z \, or\, \bar z.\eeq
By \eqref{homological equation} and \eqref{truncation}, \beq\label{functionstructure3}{\partial^2 F\over\partial z_n\partial \bar z_m}=0, |n-m|> K,{\partial^2 F\over\partial z_n\partial z_m}={\partial^2F\over\partial \bar z_n\partial \bar z_m}=0, |n+m|> K.\eeq
These restrictions are crucially used in this proof.

\noindent $\clubsuit$ The estimations of $ \{F,R\}$ and $ { \partial\over\partial I_k}\{F,R\}$.
Using the Cauchy estimate, we obtain
 \begin{eqnarray}
 |\{F,R\}|^*_{D_{\rho-\mu}(s-2\sigma,r)}&\leq& {Cr^2\over\sigma}(2d+\sum_{n\in\Z}{e^{-2|n|(\rho-\mu)}\over\la n\ra^{2\beta}}) \llbracket F\rrbracket_{r,D_\rho(s-\sigma,r)}^{\beta,\alpha,*}\llbracket R\rrbracket^{\beta,*}_{r,D_\rho(s,r)}\nonumber\\
 &\leq& \frac {Cr^2}{\sigma(\rho-\mu)^2}\llbracket F\rrbracket_{r,D_{\rho}(s-\sigma,r)}^{\beta,\alpha,*}\llbracket R\rrbracket^{\beta,*}_{r,D_{\rho}(s,r)}.
 \end{eqnarray}
Similarly, \begin{eqnarray}
 |{\partial\over\partial I_k}\{F,R\}|^*_{D_{\rho-\mu}(s-2\sigma,r)}\leq \frac {C}{\sigma(\rho-\mu)^2}\llbracket F\rrbracket_{r,D_{\rho}(s-\sigma,r)}^{\beta,\alpha,*}\llbracket R\rrbracket^{\beta,*}_{r,D(s,r)}.
 \end{eqnarray}

\noindent $\clubsuit$ The estimations of $ { \partial\over\partial z_n}\{F,R\}$ and $ { \partial\over\partial \bar z_n}\{F,R\}$.
Clearly,
$${ \partial\over\partial z_n}({\partial R\over\partial I_k}{\partial F\over\partial \theta_k})={\partial R\over\partial I_k}{\partial^2 F\over\partial z_n\partial \theta_k}+{\partial^2 R\over\partial z_n\partial I_k}{\partial F\over\partial \theta_k}.$$
We shall estimate each term one by one.

\noindent $\bullet$ Using the Cauchy estimate in $\theta_k$,
\begin{eqnarray}|{\partial R\over\partial I_k}{\partial^2 F\over\partial z_n\partial \theta_k}|^*_{D_{\rho-\mu}(s-2\sigma,r)}
&\leq& C\llbracket R\rrbracket^{\beta,*}_{r,D_{\rho}(s,r)}\cdot |{\partial F\over\partial z_n}|^*_{D_\rho(s-\sigma,r)}\sigma^{-1}\\
&\leq& {Cr\llbracket R\rrbracket^{\beta,*}_{r,D_{\rho}(s,r)}\llbracket F\rrbracket_{r,D_{\rho}(s-\sigma,r)}^{\beta,\alpha,*}e^{-|n|(\rho-\mu)}\over\sigma\la n\ra^{\beta}},\nonumber\end{eqnarray}
and using the Cauchy estimate in $I_k$ and $\theta_k$,
\begin{eqnarray}\label{firstderivative}
|{\partial^2 R\over\partial z_n\partial I_k}{\partial F\over\partial \theta_k}|^*_{D_{\rho-\mu}(s-2\sigma,r)}
&\leq& {C|{\partial R\over\partial z_n}|^*_{D_{\rho}(s,r)}\over r^2}{| F|^*_{D_\rho(s-\sigma,r)}\over\sigma}\\
&\leq& {Cr\llbracket R\rrbracket^{\beta,*}_{r,D_{\rho}(s,r)} \llbracket F\rrbracket_{r,D_{\rho}(s-\sigma,r)}^{\beta,\alpha,*}
e^{-|n|\rho(\rho-\mu)}\over\sigma\la n\ra^{\beta}}.\nonumber
\end{eqnarray}
The above two estimates yield
$$|{ \partial\over\partial z_n}({\partial R\over\partial I_k}{\partial F\over\partial \theta_k})|^*_{D_{\rho-\mu}(s-2\sigma,r)}
\leq {Cr\llbracket F\rrbracket_{r,D_{\rho}(s-\sigma,r)}^{\beta,\alpha,*}
\llbracket R\rrbracket^{\beta,*}_{r,D_{\rho}(s,r)} e^{-|n|(\rho-\mu)}\over\sigma\la n\ra^{\beta}}.$$
Similarly,
$$|{ \partial\over\partial \bar z_n}({\partial F\over\partial I_k} {\partial R\over\partial \theta_k})|^*_{D_{\rho-\mu}(s-2\sigma,r)} \leq {Cr\llbracket F\rrbracket_{r,D_{\rho}(s-\sigma,r)}^{\beta,\alpha,*}
\llbracket R\rrbracket^{\beta,*}_{r,D_{\rho}(s,r)} e^{-|n|(\rho-\mu)}\over\sigma\la n\ra^{\beta}}.$$

\noindent $\bullet$ By the Leibniz's rule and \eqref{functionstructure3}, one has
\begin{equation}
{\partial\over\partial z_n}({\partial R\over\partial z_m}
{\partial F\over\partial \bar z_m})
=\left\{\begin{array}{lc}{\partial^2 R\over\partial z_n\partial z_m}{\partial F\over\partial \bar z_m}+{\partial R\over\partial z_m}{\partial^2F\over\partial z_n\partial \bar z_n},&|m-n|\leq K,\\
{\partial R^2\over\partial z_n\partial z_m}{\partial F\over\partial \bar z_m},&|m-n|>K.
\end{array}\right.
\end{equation}
Since
\begin{eqnarray*}
|{\partial^2 R\over\partial z_n\partial z_m}{\partial F\over\partial \bar z_m}|^*_{D_{\rho-\mu}(s-2\sigma,r)}&\leq& \frac{C\llbracket R\rrbracket^{\beta,*}_{r,D_{\rho}(s,r)}e^{-|m-n|\rho}}{ \la n\ra^{\beta}\la m\ra^{\beta}}\cdot\frac{r\llbracket F\rrbracket_{r,D_{\rho}(s-\sigma,r)}^{\beta,\alpha,*}e^{-|m|(\rho-\mu)}}{\la m\ra^{\beta}}\\
&\leq&Cr\llbracket F\rrbracket_{r,D_{\rho}(s-\sigma,r)}^{\beta,\alpha,*}\llbracket R\rrbracket^{\beta,*}_{r,D_{\rho}(s,r)}{e^{-|n|(\rho-\mu)}\over
\la n\ra^{\beta}}\cdot{e^{-|m-n|\mu}\over \la m\ra^{2\beta}},
\end{eqnarray*}
if $|m|\neq |n| $, by Lemma \ref{numberinequalty}, one has
\begin{eqnarray*}
&&|{\partial R\over\partial z_m}{\partial^2 F\over\partial z_n\partial \bar z_m}|^*_{D_{\rho-\mu}(s-2\sigma,r)}\\
&\leq& \frac{Cr\llbracket R\rrbracket^{\beta,*}_{r,D_{\rho}(s,r)} e^{-|m|\rho}}{\la m\ra^{\beta}}\frac{\llbracket F\rrbracket_{r,D_{\rho}(s-\sigma,r)}^{\beta,\alpha,*}e^{-|m-n|(\rho-\mu)}}{\la n\ra^\beta \la n\ra^{\beta}||n|^\alpha-|m|^\alpha|}\\
&\leq& Cr\llbracket F\rrbracket_{r,D_{\rho}(s-\sigma,r)}^{\beta,\alpha,*}\llbracket R\rrbracket^{\beta,*}_{r,D_{\rho}(s,r)}{e^{-|n|(\rho-\mu)}\over \la n\ra^{\beta}}\cdot{e^{-|m|\mu}\over K^{1-\alpha} \la m\ra^{2\beta+\alpha-1}}\nonumber\\
&\leq& Cr\llbracket F\rrbracket_{r,D_{\rho}(s-\sigma,r)}^{\beta,\alpha,*}\llbracket R\rrbracket^{\beta,*}_{r,D_{\rho}(s,r)}{e^{-|n|(\rho-\mu)}\over \la n\ra^{\beta}}\cdot{e^{-|m|\mu}\over\la m\ra^{\beta}}.\nonumber
\end{eqnarray*}
If $m=n$, then obviously,
$$
|{\partial R\over\partial z_n}{\partial^2 F\over\partial z_n\partial \bar z_n}|^*_{D_{\rho-\mu}(s-2\sigma,r)}\leq Cr\llbracket F\rrbracket_{r,D_{\rho}(s-\sigma,r)}^{\beta,\alpha,*}\llbracket R\rrbracket^{\beta,*}_{r,D_{\rho}(s,r)} \cdot\frac{1}{\la n\ra^{3\beta}}.$$
Taking the sum in $m$, one has
 \begin{eqnarray}&&\sum_{m\in\Z}|{ \partial\over\partial z_n}({\partial R\over\partial z_m}{\partial F\over\partial \bar z_m})|^*_{D_{\rho-\mu}(s-2\sigma,r)}\\
&\leq& { r\llbracket F\rrbracket_{r,D_{\rho}(s-\sigma,r)}^{\beta,\alpha,*}\llbracket R\rrbracket^{\beta,*}_{r,D_{\rho}(s,r)}e^{-|n|(\rho-\mu)}\over \la n\ra^{\beta}}(\sum_{m\in\Z}{e^{-|m-n|\mu}\over \la m\ra^{2\beta}}+{e^{-|m|\mu}\over \la m\ra^{\beta}})\nonumber\\
&\leq& { Cr\llbracket F\rrbracket_{r,D_{\rho}(s-\sigma,r)}^{\beta,\alpha,*}\llbracket R\rrbracket^{\beta,*}_{r,D_{\rho}(s,r)}e^{-|n|(\rho-\mu)}\over \mu \la n\ra^{\beta}}.\nonumber
 \end{eqnarray}
This implies that
$$
|{\partial\over\partial z_n}\{R,F\}|^*_{D_{\rho-\mu}(r-2\sigma,s)}
\leq {CrK\llbracket F\rrbracket_{r,D_{\rho}
(s-\sigma,r)}^{\beta,\alpha,*}\llbracket R
\rrbracket^{\beta,*}_{r,D_{\rho}(s,r)}\over\mu\sigma}
{e^{-|n|(\rho-\mu)}\over \la n\ra^{\beta}}.$$

\noindent $\clubsuit$ The estimation on ${\partial^2\over\partial z_n\partial z_m}\{F,R\}$.

\noindent $\bullet$ With \eqref{functionstructure4}, one has
$$
{ \partial^2\over\partial z_n\partial \bar z_m}({\partial F\over\partial I_k}{\partial R\over\partial \theta_k})={\partial F\over\partial I_k}{\partial^3R\over\partial\theta_k\partial z_n\partial \bar z_m},$$
and therefore the estimate below is straightforward,
$$|{ \partial^2\over\partial z_n\partial \bar z_m}({\partial F\over\partial I_k}{\partial R\over\partial \theta_k})|^*_{D_{\rho-\mu}(s-2\sigma,r)}\leq {C\llbracket F\rrbracket_{r,D_{\rho}(s-\sigma,r)}^{\beta,\alpha,*}\llbracket R\rrbracket^{\beta,*}_{r,D_{\rho}(s,r)}\over\sigma}{e^{-|n-m|(\rho-\mu)}\over{\la n\ra^\beta\la m\ra^\beta}}.$$
\noindent $\bullet$ We have the expression
\begin{eqnarray}\label{derivative}
&&{\partial^2\over\partial z_n\partial \bar z_m}({\partial R\over\partial I_k}{\partial F\over\partial \theta_k})\\
&=&{\partial R\over\partial I_k}{\partial^3F\over\partial\theta_k\partial z_n\partial \bar z_m}+{\partial^2 R\over\partial I_k\partial z_n}{\partial^2F\over\partial\theta_k\partial \bar z_m}+{\partial^2 R\over\partial I_k\partial \bar z_m}{\partial^2F\over\partial\theta_k\partial z_n}+{\partial R^3\over\partial I_k\partial z_n\partial \bar z_m}{\partial F\over\partial\theta_k}.\nonumber
 \end{eqnarray}
Let $[F(\theta,z,z)]=\int_{\T^d}F(\theta,z,z)d\theta$ and recalling \eqref{secondderivative}, then by the Cauchy estimate in $\theta_k$,
\begin{eqnarray}
|{\partial R\over\partial I_k}{\partial^3F\over\partial\theta_k\partial z_n\partial \bar z_m}|^*_{D_{\rho-\mu}(s-2\sigma,r)}
&=&|{\partial R\over\partial I_k}\cdot{\partial^3 (F-[F])\over\partial\theta_k\partial z_n\partial \bar z_m}|^*_{D_{\rho-\mu}(s-2\sigma,r)}\nonumber\\
&\leq&|{\partial R\over\partial I_k}|^*_{D_{\rho}(s-\sigma,r)}|{\partial^3 (F-[F])\over\partial\theta_k\partial z_n\partial \bar z_m}|^*_{D_{\rho-\mu}(s-2\sigma,r)}\nonumber\\
&\leq& C\llbracket R\rrbracket^{\beta,*}_{r,D_{\rho}(s,r)}{\llbracket F\rrbracket_{r,D_{\rho}(s-\sigma,r)}^{\beta,\alpha,*}\over\sigma}{e^{-|n-m|(\rho-\mu)}\over{\la n\ra^\beta\la m\ra^\beta}}.\nonumber
 \end{eqnarray}
Using the Cauchy estimate in $I_k$ and $\theta_k$,
\begin{eqnarray}
&&|{\partial^2 R\over\partial I_k\partial \bar z_m}
{\partial^2F\over\partial\theta_k\partial z_n}
|^*_{D_{\rho-\mu}(s-2\sigma,r/2)}\\
&\leq&{|{\partial R\over\partial \bar z_m}|^*_{D_{\rho}(s-\sigma,r)}\over r^2}{{|{\partial F\over\partial z_n}|^*_{D_{\rho-\mu}(s-\sigma,r)}}\over \sigma}\nonumber\\
&\leq& {r\llbracket R\rrbracket^{\beta,*}_{r,D_{\rho}(s,r)}e^{-|m|(\rho-\mu)}\over{r^2\la m\ra^\beta}}\cdot{r\llbracket F\rrbracket_{r,D_{\rho}(s-\sigma,r)}^{\beta,\alpha,*}e^{-|n|(\rho-\mu)}\over \sigma{\la n\ra^\beta}}.\nonumber\\
&\leq& {C\llbracket R\rrbracket^{\beta,*}_{r,D_{\rho}(s,r)} \llbracket F\rrbracket_{r,D_{\rho}(s-\sigma,r)}^{\beta,\alpha,*} e^{-|n-m|(\rho-\mu)}\over\sigma{\la n\ra^\beta\la m\ra^\beta}}.\nonumber
 \end{eqnarray}

The estimates of ${\partial^2 R\over\partial I_k\partial z_n}{\partial^2F\over\partial\theta_k\partial \bar z_m}$ can be obtained in the same way. Finally, we consider the last function on the right-hand side of formula \eqref{derivative}. By the Cauchy estimate in $I_k$,
\begin{eqnarray}
|{\partial^3 R\over\partial I_k\partial z_n\partial \bar z_m} {\partial F\over\partial\theta_k}|^*_{D_{\rho-\mu}
(s-2\sigma,r/2)}&\leq&{|{\partial^2 R\over\partial z_n\partial \bar z_m}|^*_{D_{\rho}(s-\sigma,r)}\over r^2}
{|F|^*_{D_{\rho}(s-\sigma,r)}\over \sigma}\\
&\leq& \frac{C\llbracket R\rrbracket^{\beta,*}_{r,D_{\rho}(s,r)}
e^{-|n-m|\rho}}{r^2 \la n\ra^\beta\la m\ra^\beta}\cdot{r^2\llbracket F\rrbracket^{\beta,\alpha,*}_{r,D_{\rho}(s-\sigma,r)}\over\sigma}.\nonumber\\
&\leq& {C\llbracket F\rrbracket_{r,D_{\rho}
(s-\sigma,r)}^{\beta,\alpha,*}\llbracket R\rrbracket^{\beta,*}_{r,D_{\rho}(s,r)}e^{-|n-m|(\rho-\mu)}
\over\sigma{\la n\ra^\beta\la m\ra^\beta}}.\nonumber
 \end{eqnarray}
We conclude that
$$|{ \partial^2\over\partial z_n\partial\bar z_m}({\partial R\over\partial I_k}{\partial F\over\partial \theta_k})|^*_{D_{\rho-\mu}(s-2\sigma,r)}\leq {C\llbracket F\rrbracket_{r,D_{\rho}(s-\sigma,r)}^{\beta,\alpha,*}\llbracket R\rrbracket^{\beta,*}_{r,D_{\rho}(s,r)}\over\sigma}{e^{-|n-m|(\rho-\mu)}\over{\la n\ra^\beta\la m\ra^\beta}}.$$

\noindent $\bullet$ By \eqref{functionstructure3}, one has
 \begin{eqnarray}\label{3.3}
&&{ \partial^2\over\partial z_n\partial\bar z_m}({\partial R\over\partial z_k}{\partial F\over\partial \bar z_k})\\
&=&\left\{\begin{array}{lc}{\partial^2 R\over\partial z_n\partial z_k}{\partial^2 F\over\partial \bar z_m\partial \bar z_k}+{\partial^2 R\over\partial \bar z_m\partial z_k}{\partial^2 F\over\partial z_n\partial \bar z_k}+{\partial^3 R\over\partial z_n\partial \bar z_m\partial z_k}{\partial F\over\partial \bar z_k},& |m-k|\leq K or |n-k|<K ,\\
{\partial^3 R\over\partial z_n\partial \bar z_m\partial z_k}{\partial F\over\partial \bar z_k},&other.\end{array}
\right.\nonumber
 \end{eqnarray}
Straightforwardly, we have the estimate
 \begin{eqnarray}\label{2-order estimate 1}&&|{\partial^2R\over\partial \bar z_m\partial z_k}{\partial^2F\over\partial z_n\partial \bar z_k}|^*_{D_{\rho-\mu}(s-2\sigma,r)}\\
 &\leq& |{\partial^2R\over\partial \bar z_m\partial z_k}|^*_{D_{\rho}(r,s)}|{\partial^2F\over\partial z_n\partial \bar z_k}|^*_{D_{\rho}(r-\sigma,s)}\nonumber\\
 &\leq& {\llbracket R\rrbracket^{\beta,*}_{r,D_{\rho}(s,r)}e^{-|m-k|(\rho-\mu)}\over{\la m\ra^\beta\la k\ra^\beta}}\cdot{\llbracket F\rrbracket_{r,D_{\rho}(s-\sigma,r)}^{\beta,\alpha,*}e^{-|n-k|(\rho-\mu)}\over{\la m\ra^\beta\la k\ra^\beta}||n|^\alpha-|k|^{\alpha}|}\nonumber \\
 &\leq& {C\llbracket F\rrbracket_{r,D_{\rho}(s-\sigma,r)}^{\beta,\alpha,*}\llbracket R\rrbracket^{\beta,*}_{r,D_{\rho}(s,r)}}{e^{-|n-m|(\rho-\mu)}\over{\la n\ra^\beta\la m\ra^\beta \la k\ra^{2\beta+\alpha-1}}}\nonumber\\
 &\leq& {C\llbracket F\rrbracket_{r,D_{\rho}(s-\sigma,r)}^{\beta,\alpha,*}\llbracket R\rrbracket^{\beta,*}_{r,D_{\rho}(s,r)}}{e^{-|n-m|(\rho-\mu)}\over{\la n\ra^\beta\la m\ra^\beta}},\nonumber
 \end{eqnarray}
and
\beq\label{2-order estimate 3}|{\partial^2R\over\partial z_n\partial z_k}{\partial^2F\over\partial \bar z_m\partial \bar z_k}|^*_{D_{\rho-\mu}(s-2\sigma,r)}\leq {C\llbracket F\rrbracket_{r,D_{\rho}(s-\sigma,r)}^{\beta,\alpha,*}\llbracket R\rrbracket^{\beta,*}_{r,D_{\rho}(s,r)}}{e^{-|n-m|(\rho-\mu)}\over{\la n\ra^\beta\la m\ra^\beta}}.\eeq
Using the Cauchy estimate in $z_k$, we have
\begin{eqnarray}\label{2-order estimate 2}&&|{\partial^3R\over\partial z_n\partial \bar z_m\partial z_k}{\partial F\over\partial \bar z_k}|^*_{D_{\rho-\mu}(s-2\sigma,r/2)}\\
 &\leq& {|k|^pe^{|k|(\rho-\mu)}\over r}|{\partial^2R\over\partial z_n\partial \bar z_m}|^*_{D_{\rho}(r,s)}|{\partial F\over\partial \bar z_k}|_{D_{\rho-\mu}(r-2\sigma,s)}\nonumber\\
 &\leq& {|k|^pe^{|k|(\rho-\mu)}\over r}{\llbracket R\rrbracket^{\beta,*}_{r,D_{\rho}(s,r)}e^{-|n-m|\rho}\over{\la m\ra^\beta\la n\ra^\beta}}\cdot{r\llbracket F\rrbracket_{r,D_{\rho}(s-\sigma,r)}^{\beta,\alpha,*}e^{-|k|\rho}\over{\la k\ra^\beta}}\nonumber \\
 &\leq& {C\llbracket F\rrbracket_{r,D_{\rho}(s-\sigma,r)}^{\beta,\alpha,*}\llbracket R\rrbracket^{\beta,*}_{r,D_{\rho}(s,r)}}{e^{-|n-m|(\rho-\mu)}\over{\la n\ra^\beta\la m\ra^\beta}}e^{-|k|\mu}|k|^p.\nonumber
 \end{eqnarray}
Therefore, by \eqref{2-order estimate 1}, \eqref{2-order estimate 3}, and \eqref{2-order estimate 2}, and taking the sum in $k$, then
\beq\label{2-order estimate conclusion}
| { \partial^2\{F,R\}\over\partial z_n\partial\bar z_m}|^*_{D_{\rho-\mu}(s-2\sigma,r)} \leq{C \llbracket F\rrbracket_{r,D_{\rho}(s-\sigma,r/2)}^{\beta,\alpha,*}\llbracket R\rrbracket^{\beta,*}_{r,D_{\rho}(s,r)}\over
 \sigma\mu^{p+1}} {e^{-|n-m|(\rho-\mu)}\over \la n\ra^\beta\la m\ra^{\beta}}.
 \eeq
Similarly, one has
\begin{eqnarray}
| { \partial^2\{F,R\}\over\partial z_n\partial z_m}|^*_{D_{\rho-\mu}(s-2\sigma,r)}\leq {C \llbracket F\rrbracket_{r,D_{\rho}(s-\sigma,r)}^{\beta,\alpha,*}\llbracket R\rrbracket^{\beta,*}_{r,D_{\rho}(s,r)}\over\sigma \mu^{p+1}}{e^{-|n+m|(\rho-\mu)}\over\la n\ra^{\beta}\la m\ra^{\beta}}.
 \end{eqnarray}
The estimations of $ { \partial^2\over\partial z_n\partial\bar z_m}\{F,R\}$ and ${\partial^2\over\partial \bar z_n\partial \bar z_m}\{F,R\}$
 can be done in the same way. 
 We then have immediately our conclusion. \qed
\end{proof}

\subsection{Estimate of the New Perturbation}

Recalling that \[\llbracket P\rrbracket^{\beta,*}_{ r,D_{\rho}(s, r )}+\|X_{P }\|^*_{D(s, r )}\le
\varepsilon,\]
then by Lemma \ref{homological solution is regular} and \ref{Lem4.4}, there is a symplectic change of variables
$$\Phi_+: D_{\rho_+} (s_+, r_+)\times\mathcal O_+
 \to D_\rho(s, r ),
$$
with $s_+=s-4\sigma>0$,$r_+=\eta r$, $\eta=\epsilon^{1\over3}$, and $\rho_+=\rho-\mu>0$,
such that the vector field $X_{H\circ\Phi}$ defined on
$D_{\rho_+}(s_+, r_+)$ satisfies
\[\|X_{P_+}\|^*_{D_{\rho_+} (s_+, r_+)}\le c(\eta+e^{-K\mu})\varepsilon+
c\gamma^{-2}\mu^{-2-p} \sigma^{-d-1} K^{8\tau+2}\eta^{-4}\varepsilon e^{-K\mu}.\]
Therefore, the remaining task is that $P_+$ satisfies Assumption $\mathcal B2$. First, we have
\begin{Lemma}\label{truncation space}
Let $P\in\Gamma^{\beta}_{ r,D_{\rho}(s,r)}$ and consider its Taylor series approximation $R$(see \eqref{truncation}). Then
$$\llbracket R\rrbracket^{\beta,*}_{\eta r,D_{\rho}(s,r)}
\leq \llbracket P\rrbracket^{\beta,*}_{r,D_{\rho}(s,r)},$$
$$\llbracket P-R\rrbracket^{\beta,*}_{\eta r,D_{\rho}(s,4\eta r)}
\leq c(\eta+e^{-K\mu})\llbracket P\rrbracket^{\beta,*}_{r,D_{\rho}(s,r)}.$$
\end{Lemma}

By the Taylor series expansion, the new perturbation $P_+$ can be written as
\begin{eqnarray*}
P_+&=&P-R+\{P,F\}+\frac{1}{2!}\{\{N+\mathcal A,F\},F\}
+\frac{1}{2!}\{\{P,F\},F\}\\
&&+\cdots+\frac{1}{n!}\{\cdots\{N+\mathcal A,\underbrace{F\}
\cdots,F}_n\}
+\frac{1}{n!}\{\cdots\{P,\underbrace{F\}\cdots,F}_n\}+\cdots
\end{eqnarray*}
Since $\{N+\mathcal A,F\}=-R+\la R\ra$, by Lemma \ref{truncation space} and Lemma \ref{mainLemma}, the new perturbation $P_+$ satisfies Assumption $\mathcal B2$ with a suitable parameter setting. More precisely, with a direction computation, one has

\begin{Lemma} The new perturbation $P_+\in\Gamma^\beta_{r_+,D_{\rho_+}(s_+,r_+)}$ satisfies
$$\llbracket P_+\rrbracket^{\beta,*}
_{\eta r,D_{\rho-\mu}(s-4\sigma ,\eta r )}
\leq c(\eta+e^{-K\mu})\varepsilon+
c\gamma^{-2}\mu^{-3-p} \sigma^{-d-1} K^{8\tau+2}
\eta^{-4}\varepsilon e^{-K\mu}.$$
\end{Lemma}

\section{Iteration Lemma}
\noindent For any given positive numbers $s,r,\varepsilon,\gamma$, $\alpha$, $\beta,\rho,L,M$, and for any $\nu\ge 0$, we define the iteration sequences
\begin{eqnarray}&&\label{series}
s_{\nu+1}=s_{\nu}-\sigma_\nu,\sigma_\nu
={s\over 2^{\nu+2}},\nonumber\\
&& r_{\nu+1}=\frac{\eta_{\nu}r_{\nu}}4
=2^{-2\nu}(\prod_{i=0}^{\nu}\varepsilon_i)^{\frac 13}r_0,
\quad \eta_{\nu }=\varepsilon_\nu^{\frac 13},\\
&&\varepsilon_{\nu+1}=c\varepsilon_{\nu}(\eta_{\nu}
+e^{-K_{\nu}\mu_{\nu}}+\varepsilon_{\nu}K_{\nu}^{4\tau+2}\gamma^{-4}\mu_{\nu}^{-p-2} \sigma_{\nu}^{-d-1} ),\nonumber\\
&&M_{\nu+1}=M_{\nu}+\varepsilon_{\nu}, \quad
L_{\nu+1}=L_{\nu}+\varepsilon_{\nu},\nonumber\\
&&{K}_{\nu+1}=c\mu_{\nu+1}^{-1}\ln\varepsilon_{\nu+1}^{-1},
\quad\rho_{\nu+1}=\rho_{\nu}-\mu_{\nu},\quad
\mu_{\nu+1}={\rho\over 2^{\nu+2}},
\nonumber \end{eqnarray}
where $c$ is a positive constant, and the parameters $s_0,r_0,\varepsilon_0,\rho_0,L_0,M_0$, and $K_0$ are defined as $s, r,\varepsilon,\rho,L,M$ and $c\mu^{-1}\ln \frac{1}{\varepsilon}$, respectively. With the notation $D_\nu=D_{\rho_\nu}(s_\nu,r_\nu)$, we have
\begin{Lemma}\label{Lem5.1}
Let $\varepsilon_0$ be small enough and $\nu\ge 0$. Suppose that

\noindent $(1)$\quad $N_\nu+\mathcal A_\nu=\la\omega_\nu,I\ra+\sum\limits_{n\in\Z}\Omega_n^\nu z_n\bar z_n+\sum\limits_{|n|\leq K_{\nu-1}}a^{\nu}_{n,-n}(\xi)z_n\bar z_{-n}$ is a normal form with parameters
$\xi$ on a closed set $\mathcal O_{\nu}$ of $\R^d$.
For any $\xi\in\mathcal O$, $|k|\leq K_\nu$ and $n,m\in\Z$ with $|n\pm m|\leq K_\nu$, there are 
\begin{eqnarray*}
&&|\langle k,\omega_\nu(\xi)\rangle^{-1}|< {\gamma\over K_\nu^{\tau}},\,k\neq 0\nonumber\\
&&\|(\langle k,\omega_\nu\rangle {\mathbb I}_n+A_n^\nu)^{-1}\|<{\gamma\over K_\nu^{2\tau}},\\
&&\|(\langle k,\omega_\nu\rangle {\mathbb I}_{nm} \pm( A_{n}^\nu\otimes {\mathbb I}_n+{\mathbb I}_m\otimes A_{m}^\nu))^{-1}\|<{\gamma\over K_\nu^{4\tau}},\nonumber\\
&&\|(\langle k,\omega_\nu\rangle {\mathbb I}_{nm} \pm( A_{n}^\nu\otimes {\mathbb I}_n-{\mathbb I}_m\otimes A_{m}^\nu))^{-1}\|< {\gamma\over K_\nu^{4\tau}},\,k\neq0\&|n-m|<K,\nonumber
\end{eqnarray*}
where
$$A_0^\nu =\Omega_0^\nu, \quad
A_n^\nu= \left(\begin{array}{cc}\Omega^\nu_n&a_n^\nu \\
a_{-n,n}^\nu & \Omega_{-n}^\nu\end{array}\right), |n|\geq 1.$$

\noindent $(2)$\quad $\omega_\nu(\xi)$, $\Omega_{n}^\nu(\xi)$ are
Lipschitz in $\xi$ and satisfy
$$
|\omega_\nu-\omega_{\nu-1}|^*_{\mathcal O_{\nu}}\le
\varepsilon_{\nu-1}, \quad ||n|^{2\beta}(\Omega_n^\nu-\Omega_n^{\nu-1})|^*_{\mathcal
O_{\nu}}\le\varepsilon_{\nu-1};$$

\noindent $(3)$\quad $N_\nu+\mathcal A_\nu+P_\nu $ satisfies Assumption $\mathcal A,\mathcal B$ with $r_\nu,s_\nu,\rho_\nu,\varepsilon_\nu$ and
\[\llbracket P_\nu\rrbracket^{\beta,*}_{ r_\nu,D_\nu}
+\|X_{P_\nu}\|^*_{r_\nu,D_\nu}\le\varepsilon_\nu.\]
Then, there exists a new closed set $\mathcal O_{\nu+1}=:\mathcal
O_\nu\setminus\mathcal R^{\nu+1} $ $($ see $(\ref{resonant})$ for the construction of $\mathcal R^{\nu+1}$ $)$, and a symplectic transformation of variables,
\beq\Phi_\nu:D_{\nu+1} \times\mathcal O_{\nu}\to D_{\nu+1},\eeq
such that on $D_{\nu+1}\times\mathcal O_{\nu}$, $H_{\nu+1}=H_\nu\circ\Phi_\nu$ takes the form
\beq H_{\nu+1}=\la\omega_{\nu+1},I\ra+\sum\limits_{n\in\Z} \Omega^{\nu+1}_n z_n\bar z_n+\sum\limits_{|n|\leq K_{\nu+1}}a^{\nu+1}_{n,-n}(\xi)z_n\bar z_{-n}+P_{\nu+1}. \label{5.4}\eeq
The Hamiltonian $H_{\nu+1}$ satisfies all the assumptions of $ H_\nu$ with $\nu+1$ in place of $\nu$.
\end{Lemma}

\section{Convergence}
We follow the proofs in \cite{GT} and \cite{EK1}. First, we have estimates,
\begin{Lemma}\label{youyongdeguji}
For $\nu\geq 0$ and $n\in\Z$,
$${1\over\sigma_\nu}|\Phi_{\nu+1}-id|_{r_{\nu},D_{\nu+1}}^*,|D\Phi_{\nu+1}-Id|_{r_\nu,r_{\nu},D_{\nu+1}}^*\leq c\gamma^{-4}\mu_\nu^{-2}\sigma_\nu^{-d-1}K_\nu^{4\tau+2}\varepsilon_\nu.$$
$$|\omega_{\nu+1}-\omega_{\nu}|_{\mathcal O_{\nu}}^*
\leq\varepsilon_{\nu},\,\sup_{n\in\Z}||n|^{2\beta}(\Omega_n^{\nu+1}-\Omega_n^{\nu})|^*_{\mathcal O_{\nu}}\le
\varepsilon_{\nu},\,\sup_{n\in\Z} e^{|n|\rho_{\nu+1}}
|n|^{2\beta}|a_{n,-n}^{\nu+1}|_{\mathcal O_{\nu}}^*
\leq \varepsilon_{\nu}. $$
\end{Lemma}

To apply Lemma \ref{Lem5.1} when $\nu=0$, we set $\varepsilon_0=\varepsilon,r_0=r, s_0=s, \rho_0=\rho, L_0=L$, $N_0=N,\mathcal A_0=0, P_0=P$. The smallness conditions are satisfied if we set $\varepsilon_0$ sufficiently small. The small divisor conditions are satisfied by setting $\mathcal O_1=\mathcal O\backslash \mathcal R^0$(see \eqref{resonant}). Then the iterative Lemma applies,
we obtain a sequence of transformations $\Psi^\nu$ defined on $D_{\nu+1}\times \mathcal O_{\nu+1}$ with
\[\Psi^\nu=\Phi_0\circ\Phi_1\circ\cdots\circ\Phi_\nu:D_{\nu+1}
\times\mathcal O_{\nu+1}\to D(r_0,s_0),\nu\ge 0,\]
such that $H\circ\Psi^\nu=N_{\nu+1}+P_{\nu+1}$. For $\nu\geq 0$, by the chain rule, we have
\begin{equation}\pmb{\pmb |}D\Phi^{\nu+1} \pmb{\pmb |}_{r_0,r_{\nu+1},D_{\nu+1}}\leq \prod\limits_{m=1}^{\nu+1} \pmb{\pmb |}D\Phi_m \pmb{\pmb |}_{r_{m-1},r_m,D_m}\leq \prod\limits_{m=1}^{\nu+1}(1+\epsilon_{m-1}^{1\over2})\leq 2.\end{equation}

Therefore, with the mean-value theorem, we obtain
\begin{equation}|\Psi^{\nu+1}-\Psi^{\nu}|_{r_0,D_{\nu+1}}
\leq \pmb{\pmb |}D\Psi^{\nu} \pmb{\pmb |}_{r_0,r_\nu,D_\nu}|\Phi_{\nu+1}-id|_{r_{\nu},D_{\nu+1}}\leq 2\epsilon_{\nu}^{2\over3}\nonumber,\end{equation}
and $\Psi^\nu$ converges uniformly to $\Psi^\infty$ on $D_{\frac 12\rho}(\frac12r,0)\times\mathcal O_\gamma$
We have estimate \eqref{transfromation} on $D_{\frac 12\rho}(\frac 12r,0)\times\mathcal O_\gamma$ with $\mathcal O_\gamma=\bigcap_{\nu\geq 1}\mathcal O_{\nu}$.

It remains to prove that $\Psi^\infty$ is indeed defined on $D_{\rho\over 2}({s\over 2},{r\over 2})\times\mathcal O_\gamma$ with the same estimates. A similar discussion in \cite{GT} indicates that the estimate \eqref{transfromation} can be extended to the domain $D_{\rho\over 2}({s\over 2},{r\over 2})$. The estimates \eqref{frequencyshift} are simple and hence we omit the details.

Note that $H$ is analytic on $D_{\rho\over 2}({s\over 2},{r\over 2})$, we deduce that $H\circ\Psi^\infty=N^*+\mathcal A^*+P^*$ is analytic on $D_{\rho\over 2}({s\over 2},{r\over 2})$. Finally, we need to prove that
$$\partial_y P^* = \partial_z P^* = \partial_{\bar z}P^* = 0, \partial_{z_iz_j}^2
P^* = \partial_{z_i \bar z_j}^2
P^* = \partial_{\bar z_i \bar z_j}^2
P^* = 0$$
on $D_{\rho\over 2}({s\over 2},{r\over 2})\times\mathcal O_\gamma$. In the following, we only give the proof for $\partial_{z_iz_j}^2
P^* = 0$; the others can be treated in the same way. Note that $\|\partial_{z_iz_j}^2 P_\nu\|_{D({s\over 2})}\leq \epsilon_\nu$ and $\|\partial_{z_iz_j}^2(P_\nu-P_{\nu+1})\|_{D({s\over 2})}\leq \epsilon_\nu+\epsilon_{\nu+1}$. It follows that
$$\|\partial_{z_iz_j}^2(P_\nu-P^*)\|_{D({s\over 2})}
\leq \sum_{k=\nu}^\infty\|\partial_{z_iz_j}^2
(P_\nu-P_{\nu+1})\|_{D({s\over 2})}\leq 2\epsilon_\nu$$
and then
$$\|\partial_{z_iz_j}^2
P^*\|_{D(s/2)}\leq \|\partial_{z_iz_j}^2
P_\nu\|_{D(s/2)}+\|\partial_{z_iz_j}^2
(P_\nu-P^*)\|_{D(s/2)}\leq 4\epsilon_\nu$$
for all $\nu\geq 0$, this means $\partial_{z_iz_j}^2 P^*=0$ on $D_{\rho\over 2}({s\over 2},{r\over 2})\times\mathcal O_\gamma$.

%

\section{Measure Estimates}\label{measure}
By \eqref{tau}, we have $\tau_1>d+3+{4\over\alpha^2},\varsigma={\tau_1+1\over 1-\alpha}.$
For any $\nu\geq 0$, we define ${\mathcal O}_{\nu+1}=\mathcal O_{\nu}\backslash \mathcal R^{\nu}$, the resonance set $\mathcal
R^{\nu}$ is defined to be
\begin{equation}\label{resonant}\mathcal R^{\nu}=\mathcal R^{\nu,0}\bigcup\mathcal R^{\nu,1}\bigcup\mathcal R^{\nu,2}\bigcup \mathcal R^{\nu,11},\end{equation}
where
\begin{eqnarray*}
\mathcal R^{\nu,0}&=&\bigcup_{0<|k|\le K_{\nu}}\mathcal R_k^{\nu,0}=\bigcup_{0<|k|\le K_{\nu}}\{\xi\in \mathcal
O_{\nu-1}:|\la
k,\omega_{\nu}(\xi)\ra^{-1}|\geq\frac{K_{\nu}^{\tau_1}}{\gamma}\},\label{6.1}\\
\mathcal R^{\nu,1}&=&\bigcup_{|k|\le K_{\nu},\atop n\in \Z}\mathcal R_{kn}^{\nu,1}=\bigcup_{|k|\le K_{\nu},\atop n\in \Z}\{\xi\in \mathcal
O_{\nu-1}:\|(\langle k,\omega_{\nu}\rangle {\mathbb I}_n+A_n^{\nu})^{-1}\|\geq{K_\nu^{2\tau_1}\over \gamma}\},\label{6.2}\\
\mathcal R^{\nu,2}&=&\bigcup_{|k|\le K_{\nu},\atop n,m\in \Z}\mathcal R_{knm}^{\nu,2}\label{22m}\\
&=&\bigcup_{|k|\le K_{\nu},\atop n,m\in \Z}\{\xi\in \mathcal
O_{\nu-1}:\|(\langle k,\omega_{\nu}\rangle {\mathbb I}_{nm} \pm( A_{n}^{\nu}\otimes {\mathbb I}_n+{\mathbb I}_m\otimes A_{m}^{\nu}))^{-1}\|\geq {K_\nu^{4\tau_1}\over \gamma}\},\nonumber\\
\mathcal R^{\nu,11}&=&\bigcup_{0<|k|\le K_{\nu},\atop|n-m|\leq K } \mathcal R_{knm}^{\nu,11}\label{6.3}\\
&=&\bigcup_{0<|k|\le K_{\nu},\atop |n-m|\leq K } \{\xi\in \mathcal O_{\nu-1}:\|(\langle k,\omega_{\nu}\rangle {\mathbb I}_{nm} \pm( A_{n}^{\nu}\otimes{\mathbb I}_n-{\mathbb I}_m\otimes A_{m}^{\nu}))^{-1}\|\geq {K_\nu^{12\tau_1+16\sigma}\over \gamma}\}.\nonumber
\end{eqnarray*}

\begin{Lemma}\label{Lem6.1}
$${\rm meas}(\mathcal R^{\nu,0})\leq\frac {\gamma}{K_{\nu}^{\tau_1-d}},\,{\rm meas}(\mathcal R^{\nu,1})\leq\frac {\gamma^{1\over2}}{K_{\nu}^{{\tau_1-d-{2\over\alpha}}}},\,
{\rm meas}(\mathcal R^{\nu,2})\leq\frac {\gamma^{\frac 14}}{K_{\nu}^{{\tau_1-d-{4\over\alpha^2}}}}.$$
\end{Lemma}
The proof of this Lemma is standard and is omitted.

\begin{Lemma}({\it Lemma 7.6 of \cite{CY}})\label{Lem6.2}
Let $M$ be a $N\times N$ non-singular matrix with $\|M\|<B$; then,
$$\{\omega:\|M^{-1}\|\geq h\}\subset\{\omega:|det M|
<{cB^{N-1}\over h}\}.$$
\end{Lemma}
\begin{Lemma}
$${\rm meas}(\mathcal R^{\nu,11})
\leq\frac {\gamma^{\frac 14}}{K_{\nu}^{\tau_1}}.$$
\end{Lemma}
\begin{proof}Recalling the truncation $R_\nu$ in \eqref{truncation} and the homological equation \eqref{homological equation}, one has $0< |k|\leq K_\nu$ and $|n-m|\leq K_\nu$. Because $\alpha<1$, then $||n|^\alpha-|m|^\alpha|\leq K_\nu$ and hence
$$\|\langle k,\omega_{\nu}\rangle I_{nm} \pm( A_{n}^{\nu}\otimes{\mathbb I}_n
-{\mathbb I}_m\otimes A_{m}^{\nu})\|\leq CK_\nu.$$
Then, by Lemma \ref{Lem6.2},
\begin{eqnarray*}&&\mathcal R_{knm}^{\nu,11}
\subset\mathcal Q^{\nu,11}_{knm}\\
&=&\{\xi\in \mathcal O_{\nu-1}:
\|det(\langle k,\omega_{\nu}\rangle {\mathbb I}_{nm} \pm( A_{n}^{\nu}\otimes{\mathbb I}_n-{\mathbb I}_m\otimes A_{m}^{\nu}))\|
\leq {K_\nu^{12\tau_1+16\sigma-3}\over \gamma}\}.\end{eqnarray*}

Let $a=m-n$, then
$$\bigcup\limits_{0<|k|\le K_{\nu},\atop n,m\in\Z}\mathcal Q_{knm}^{\nu,11}=\bigcup\limits_{0<|k|\le K_{\nu},\atop n\in\Z,|a|\leq K_\nu}\mathcal Q_{k,n,n+a}^{\nu,11}.$$

By Lemma \ref{Lem6.1}, for any $\xi\in\mathcal R^{\nu,0}$ and
$ 0<|k|\leq K_{\nu}$, one has $$|\langle k,\omega\rangle|\geq \gamma K_{\nu}^{-\tau_1}.$$
Now we will prove $\mathcal  Q_{knm}^{\nu,11}=\emptyset$ if $|k|,|n-m|\leq K_\nu$ and  $\max\{|n|,|m|\}\geq K_{\nu}^{\tau_1+2\varsigma}$. For the set with such restrictions, one has $|n|,|m|\geq K_{\nu}^{\tau_1+2\varsigma-1}$ by Lemma \ref{numberinequalty}. Let $a=m-n$, then $|a|\leq K_\nu$.
Note that  $\varsigma={\tau_1+1\over 1-\alpha}$, $\alpha+\beta\geq1$ and $\varepsilon_0<e^{-{4\rho \over\gamma}}$, there is
\begin{eqnarray*}
&&|\langle k,\omega\rangle+\Omega_{n}^{\nu}-\Omega_{n+a}^{\nu}|\\
&=&|\langle k,\omega\rangle+|n|^\alpha
+\tilde\Omega_n^{\nu}-|n+a|^\alpha-\tilde \Omega_{n+a}^{\nu}|\\
&\geq&|\langle k,\omega\rangle|-||n|^\alpha-|n+a|^\alpha|
-|\tilde\Omega_n^{\nu}|-|\tilde \Omega_{n+a}^{\nu}|\\
&\geq& \gamma K_{\nu}^{-\tau_1}-{\alpha|a|\over |n|^{1-\alpha}}
-{\varepsilon_0\over |n|^{2\beta}}
-{\varepsilon_0\over |n+a|^{ 2\beta}}\\
&\geq& \gamma K_{\nu}^{-\tau_1}
-{\alpha|a|\over K_{\nu}^{(\tau_1+2\varsigma-1)(1-\alpha)}}
-{2\varepsilon_0\over {K_{\nu}^{2(\tau_1+2\varsigma-1)\beta}}}\\
&\geq& \gamma K_{\nu}^{-\tau_1}-{\gamma\over4}K_{\nu}^{-\tau_1}
-{\gamma\over4}K_{\nu}^{-\tau}\\
&\geq& {1\over2}\gamma K_{\nu}^{-\tau_1}.
\end{eqnarray*}
 By \eqref{normalform}, one has
\begin{eqnarray}
|\det(\langle k,\omega\rangle {\mathbb I}_{nm} \pm( A_n^\nu\otimes {\mathbb I}_n-{\mathbb I}_m\otimes A_m^\nu))|
\geq {1\over32}\gamma^4 K_{\nu}^{-4\tau_1}.
\end{eqnarray}
Thus, we have following
$$\mathcal  R^{\nu,11}\subset\bigcup\limits_{0<|k|\le K_{\nu},\atop |n-m|\leq K_\nu}\mathcal  Q_{knm}^{\nu,11}=\bigcup_{ 0<|k|\le K_{\nu},|n-m|\leq K_\nu\atop |n|,|m|\leq K_\nu^{\tau_1+2\varsigma}}\mathcal  Q_{knm}^{\nu,11
}.$$

Let $$M=\det(\langle k,\omega_\nu\rangle {\mathbb I}_{nm} \pm( A_{n}^\nu\otimes {\mathbb I}_n-{\mathbb I}_m\otimes A_{m}^\nu)),$$ and then with a simple computation, one has $$\inf\limits_{\xi\in\mathcal O}\max\limits_{0<d\leq
4}|{\partial_\xi^{d}M}|\geq {1\over 2}|k|^4.$$
In view of Lemma $\ref{Lemma600}$, we have
$${\rm meas}(\mathcal Q_{knm}^{\nu,11})
\leq\frac {\gamma^{\frac 14}}{K_{\nu}^{3\tau_1+4\varsigma-1}},$$
and then
\begin{eqnarray*}
{\rm meas}(\mathcal R^{\nu,11})
\le {\gamma^{\frac 14}\over K^{3\tau_1+4\varsigma-1}}*
K^{2\tau_1+4\varsigma}*K_\nu^d
\le{\gamma^{\frac 14}\over {K_\nu^{\tau_1-d-1}}}.
\end{eqnarray*}
\end{proof}

\begin{Lemma}\label{Lem6.3} Let $\tau_1>d+3+{4\over\alpha^2}$; then the total measure needed to be excluded in the KAM iteration is
\begin{eqnarray*}
{\rm meas}(\bigcup_{\nu\ge 0}\mathcal R^{\nu})
\leq{\rm meas}[\mathcal R^{\nu,0}\bigcup\mathcal R^{\nu,1}
\bigcup\mathcal R^{\nu,2}\bigcup \mathcal R^{\nu,11}]
\le\sum_{\nu\ge 0}\frac{\gamma^{1\over4}}{K_\nu^{\tau_1-d-1}}
\le\gamma^{1\over4}.
\end{eqnarray*}
\end{Lemma}

\appendix
\section{Appendix}

\begin{Lemma}\label{numberinequalty}
For $K>1$ and any $n,m\in\Z\backslash\{0\}$ such that $n\neq m$ and $|n-m|\leq K$, one has $${|m|\over K}\leq |n|\leq K |m| $$ and
$$||n|^\alpha-|m|^\alpha|\geq {\alpha\over2 |m|^{1-\alpha}}.$$
\end{Lemma}

\begin{Lemma}\label{Lemma600}(Lemma 8.4 of \cite {Ba3}).
Let $g:\mathcal I\to \mathbb R$ be $b+3$-times differentiable, and assume that

\noindent (1) $\forall \sigma\in\mathcal I$, there exists $s\le b+2$ such that $g^{(s)}(\sigma)>B$.

\noindent (2) There exists $A$ such that $|g^{(s)}(\sigma)|\le A$ for $\forall \sigma\in\mathcal I$ and $\forall s$ with $1\le s\le b+3$.

Define $$\mathcal I_h\equiv\{\sigma\in\mathcal I:|g(\sigma)|\le h\}, $$ then
$$\frac{{\rm meas}(\mathcal I_h)}{{\rm meas}(\mathcal I)}
\le\frac{A}{B}2(2+3+\cdots+(b+3)+2B^{-1})h^{\frac{1}{b+3}}.$$
\end{Lemma}

\section*{Acknowledgements}
The author was supported by the NSFC, Grant No. 11771077. The author would like to express his sincere gratitude to Professor M.Gao for her valuable suggestions and reading the manuscript.

\end{document}